\documentclass[11pt]{amsart}
\usepackage[T1]{fontenc}
\usepackage{amssymb,amsopn,amsmath,amsthm,eucal,hyperref,mathrsfs,array,graphicx,xcolor,bbm,enumerate,wasysym}
\usepackage{verbatim}
\usepackage[all]{xy}
\usepackage{fullpage}

\newtheorem{thm}{Theorem}[section]
\newtheorem{lemma}[thm]{Lemma}
\newtheorem{rem}[thm]{Remark}
\newtheorem{defn}[thm]{Definition}
\newtheorem{claim}[thm]{Claim}
\newtheorem{prop}[thm]{Proposition}
\newtheorem{cor}[thm]{Corollary}

\newcommand{\Q}{\mathbb Q}

\newcommand{\R}{\mathbb R}

\newcommand{\N}{\mathbb N}

\newcommand{\C}{\mathcal C}
\newcommand{\G}{\mathcal G}

\newcommand{\D}{\mathbb D}
\newcommand{\E}{\mathbb E}

\newcommand{\CR}{\mathsf{CR}}

\newcommand{\I}[1]{\mathbf{1}_{\left \{#1\right \}}}
\renewcommand{\P}{\mathbb P}

\newcommand{\Leb}{\mathsf{Leb}}

\newcommand{\eps}{\epsilon}

\newcommand{\dist}{\mathsf{dist}}

\newcommand{\radS}{\mathsf{Rad}}

\begin{document}
	\title{Gaussian multiplicative chaos through the lens of the 2D Gaussian free field}
	\author{Juhan Aru}
	\address {
		Institute of Mathematics,
		EPFL, CH-1015 Lausanne,
		Switzerland}
		\email
	{juhan.aru@math.epfl.ch}
	\begin{abstract}
	The aim of this review-style paper is to provide a concise, self-contained and unified presentation of the construction and main properties of Gaussian multiplicative chaos (GMC) measures for log-correlated fields in 2D in the subcritical regime. By considering the case of the 2D Gaussian free field, we review convergence, uniqueness and characterisations of the measures; revisit Kahane's convexity inequalities and existence and scaling of moments; discuss the measurability of the underlying field with respect to the GMC measure and present a KPZ relation for scaling exponents.
	\end{abstract}
	\maketitle

Gaussian multiplicative chaos (GMC) measures give a mathematical meaning to the exponential of rough Gaussian fields, that are not defined pointwise. As these fields themselves can be often seen as sums of independent Gaussians, their exponentials can be seen as products of independent Gaussians - hence also the name multiplicative chaos. Whereas the exponential of the 2D Gaussian free field had already been defined before in \cite{HK}, the article of Kahane ``Sur le chaos multiplicatif'' \cite{Kah} was the real starting point for a more thorough study of GMC measures. 

Over the past 10 years or so, there has been revived interest in GMC measures for log-correlated fields in 2D (e.g. the 2D continuum Gaussian free field) due to its connections with 2D Liouville quantum gravity (LQG) \cite{DS, DMS, DKRV}. More precisely, one of the first steps towards a mathematical description of LQG is to give a meaning to $\exp(\gamma \Gamma)$ where $\Gamma$ is a 2D Gaussian field with covariance of the form $\C_\Gamma(z,w) = -\log |z-w| + g(z,w)$ and $\gamma$ is a ``coupling'' parameter. It comes out that one can indeed give it a meaning in terms of random measures: one first approximates the underlying field by point-wise defined fields, then takes their exponentials and finally shows that suitably renormalized versions of these exponentials converge to a random measure.

There is already a comprehensive review of the results and applications of these GMC measures in general and in particular for log-correlated fields by Rhodes and Vargas \cite{RVrev}. Here, our aim and scope are slightly different: we would like to give a concise and self-contained treatment of the subcritical regime of GMC measures for log-correlated fields. We have tried to bring together in a unified, compact, and applicable manner various aspects of the theory, that have been previously presented in different articles with diverse notations, styles, or even languages. We hope that in this way we are able to better highlight that many proofs share similar underlying techniques and ideas. 

We would like to stress that there are no new results in this paper. Also, most of the core ideas in the article either stem from the existing literature, or (for those that are hard to track down in the existing literature) can be considered as folklore in the field. 

After the preliminaries the article is structured so that each section corresponds to roughly one single batch of conceptual ideas:
\begin{description}
\item [Section \ref{sec:basic}] discusses basic constructions of the GMC measures for log-correlated fields through approximations by pointwise defined fields and measures in the $L^2$ regime, i.e. in the regime where the second moment of the GMC mass of the domain exists. The key components are simple second moment calculations and convergence theorems for regular martingales. Via the same methods, we also discuss why all the different constructions give the same measure and show how one can simultaneously construct GMC measures over a whole interval of the ``coupling'' parameter $\gamma$.
\item [Section \ref{sec:RM}] studies the rooted GMC measure: i.e. the measure on pairs (GMC measure, random point from the GMC measure). Rooted measures are analogous to looking at the spine particle in branching random walks, and turn out to be very helpful: 
\begin{itemize}
	\item they provide the natural setting for describing how the measure / underlying field looks around a ``typical'' point; \item give a way to characterise the GMC measure; 
	\item and help extend second moment arguments from the $L^2$ to the $L^1$ regime by observing that especially rough parts of the measure cause the blow-up of the second moment of the mass, yet contribute so little to the measure that they can be safely removed.
\end{itemize}
\item [Section \ref{sec:intr}] looks in more detail at the way of characterizing the GMC measure directly as an exponential of the underlying Gaussian field. This viewpoint is not only intrinsically appealing, but also helps to prove the uniqueness of the measure for even a larger family of approximations. We also discuss another constructions of GMC measures in the $L^1$ regime by defining GMC measures on top of GMC measures.
\item [Section \ref{sec:scaling}] uses the Markov decomposition of the GFF to describe the measure of a ball of radius $r$: the measure of a ball of radius $r$ is roughly given by a scaled copy of the measure, times $\exp(\gamma X_r- \frac{\gamma^2}{2}\E [X_r^2])$ where $X_r$ is an independent Gaussian shift of variance $-\log r$, stemming from the $r$-circle average process of the underlying field. From this expression deduce scaling relations for the GMC mass of small balls, exhibiting a ``multifractal'' behaviour and prove the existence of all negative moments of the measure.
\item [Section \ref{sec:meas}] builds on the heuristic from the last section and makes it more precise: we will see that the field given by the logarithm of the GMC mass of $r-$balls around each point is up to an additive shift in the sense of distributions very close to the $r$-circle-average field of the GFF. As this approximation becomes precise in the $r \to 0$ limit, we obtain that the underlying Gaussian field is measurable w.r.t. the GMC measure of a fixed parameter.
\item [Section \ref{sec:Kah}] contains the Kahane's convexity inequality that corresponds vaguely to Slepian inequalities and gives a way to compare two GMC measures by comparing the covariance kernels of the underlying fields. We then use this inequality to derive the existence of positive moments of the measure by using scaling relations and near-independence of distance parts. 
\item [Section \ref{sec:KPZ}] finally discusses the KPZ relation which describes the size of fractals seen from the point of view of the GMC measure. We have chosen the version of the KPZ relation introduced in \cite{DS}. In this version the methods of other chapters come nicely together: we need to study in even greater detail how the measure at small scales corresponds to the behaviour of the exponential of the underlying field.
\end{description}
To keep things unified and friendly, we have opted to present everything in the case where the log-correlated field is precisely the 2D Gaussian free field. This is the case which has been of most interest in the recent years due to the link to the probabilistic framework of Liouville quantum gravity, and it presents some advantages due to its Markovian structure and exact calculations. Most of the main results generalize easily to all (sufficiently nice) log-correlated fields in 2D via Kahane's convexity inequalities. The generalization may not be as direct only in the cases of the measurability of the underlying field for a fixed parameter in Section \ref{sec:meas} and the exact version of the KPZ relation presented in Section \ref{sec:KPZ}.

We do not touch upon the applications of the GMC measures, and in particular its applications in the probabilistic theory of Liouville quantum gravity. The interested reader will get the first feeling from either the lecture notes \cite{Berln, RVln}, or the Bourbaki seminar notes on the KPZ relation \cite{Gar}. We do not also touch upon the strong links in results and techniques between Gaussian multiplicative chaos and multiplicative cascades, but would like to draw the attention to \cite{APS}, where it is explained how the GMC measure of the GFF can be exactly seen as a certain multiplicative cascade.
 
Finally, in the spirit of self-containedness and for the sake of readability we have opted to leave out references in the main text. The links to the original literature are brought out separately at the end of each section. The list of references here is by no means exhaustive, we just try to pin down the origin of the specific ideas or arguments used. For a wider and more thorough list of references see e.g. \cite{RVrev, DMS}.

\subsection*{Acknowledgements}
This short note has its roots in a reading group on GMC at ETH Z\"urich and University of Z\"urich that took place in Zurich in fall 2016 and which I had the pleasure to co-organize. I am indebted to all participants of the reading group, and in particular I would like to thank A. Sep\'{u}lveda for several discussions on GMC theory. The article was further motivated by the conference Inhomogeneous Random Systems (IRS) 2017 in Paris: I would like to thank the organizers (F. Dunlop, T. Gobron, E. Saada) and the moderator B. Duplantier for the possibility to present a talk and the encouragement to write this note. Finally, I am thankful to E. Powell, F. Viklund and the anonymous referee for very helpful comments on the draft, M. Forde for his questions and comments, and to W. Werner for some good advice. 



%

\section*{A few preliminaries}

Let us first present a few general preliminaries. For simplicity we will work with bounded, simply-connected domains $D$, although neither of those is an actual restriction and everything generalizes, if properly stated, to unbounded and at least finitely-connected case. 

\subsection{The 2D continuum Gaussian free field}\label{sec:prelim}

The zero boundary 2D continuum Gaussian free field (GFF), denoted by $\Gamma^D$, is the Gaussian process on $D$ whose covariance $\C_{\Gamma^D}(z,w)$ is given by the zero boundary Green's kernel $$\G_D(z,w) = -\log |z-w| + g(z,w).$$
Here $g(z,\cdot)$ is the harmonic extension of $\log |z-\cdot|$ on the boundary. Notice that in the context of Gaussian multiplicative chaos the normalization (i.e. the factor in front of the $\log$) is chosen for convenience to be $1$, differing from a part of the GFF/SLE literature where it is natural to set it equal to $(2\pi)^{-1}$. 

As the Green's function explodes on the diagonal, the GFF does not make direct sense as a pointwise defined Gaussian field, but it can be seen as a random Schwartz distribution. This Gaussian random distribution is characterized by saying that 
\begin{itemize}
\item For any compactly supported test function $f$ on $D$, the random variable $(\Gamma^D,f)$ is a centered Gaussian of variance
 $$ \E [(\Gamma,f) (\Gamma,f)]  =  \iint_{D\times D} f(x) \G_D(x,y) f(y) d x d y. $$  
\end{itemize}

Its existence follows for example from a direct construction:
Let $(\phi_i)_{i\geq 1}$ be the eigenfunctions of the Dirichlet Laplacian on $D$ normalized such that their Dirichlet norm $\| \phi_i \|_\nabla := \| (2\pi)^{-1}\nabla \phi_i \|_2 = 1$. Then if $X_i$ are i.i.d. standard Gaussians, the random sums $\Gamma^n = \sum_{i=1}^n X_i \phi_i$ converge in law to a random distribution $\Gamma^D$ satisfying the condition above. Convergence here takes place for example in the space of distributions, and more precisely in the Sobolev space $\mathcal{H}^{-\eps}(D)$ for any $\eps > 0$. For us this best possible analytic space does not matter too much and we refer to for example to introductory texts \cite{Gar,Berln} or \cite{She} for a more thorough discussion on the GFF. What does matter is the fact that the Cameron-Martin space of $\Gamma$ is $\mathcal{H}_0^{1}$: the closure of smooth compact functions on $D$ w.r.t the norm $\|\phi\|_\nabla$. 

We are often interested in departing from the GFF itself and then approximating it in a measurable way. Here, for the sake of circle-average and mollifier approximations it is helpful to extend the GFF by $0$ outside of $D$.

\begin{thm}[GFF approximations]\label{thm:GFFapp}
The GFF $\Gamma^D$ is the limit of the following continuous Gaussian fields:
\begin{enumerate}
\item (\textit{Orthogonal decomposition}) Let $(\phi_i)_{i\geq 1}$ be the eigenfunctions of the Dirichlet Laplacian on $D$ normalized such that $\| (2\pi)^{-1}\nabla \phi_i \|_2 = 1$. Then $X_i = (\Gamma^D, \phi_i)_\nabla$ are i.i.d standard Gaussians and moreover $\Gamma^n := \sum_{i=1}^n X_i \phi_i$ converge to $\Gamma^D$ almost surely in the space of distributions.
\item (\textit{Circle average approximation}) Let $\rho_\eps^z$ be the unit uniform measure on the circle of radius $\eps$ around $z$. Then $\Gamma_\eps(z) := (\Gamma^D, \rho_\eps^z)$ converges in probability to $\Gamma^D$ in the space of distributions.
\item (\textit{Smooth mollifiers})For any mollifier $\rho_\eps^z(w) = \rho(\frac{w - z}{\eps})$ where $\rho$ is some radially symmetric positive smooth function of unit mass and support in $\D$, we can define $\Gamma_{\eps,\rho}(z) := (\Gamma^D, \rho_\eps^z)$. Then $\Gamma_{\eps,\rho}(z)$ converges in probability to $\Gamma^D$ in the space of distributions.
\end{enumerate}
\end{thm}
These convergences can be for example easily verified by checking:
\begin{itemize}
\item Finite dimensional convergence, i.e. we check that for any smooth $f$,  converges to a centred Gaussian of variance $\iint_{D\times D} f(x) \G_D(x,y) f(y) d x d y;$. In the case of the orthogonal sequence, this follows from the fact that $(\Gamma^n,f)$ is a $L^2$ martingale w.r.t the filtration generated by $\Gamma^n$; in the other cases, one can for example show that $(\Gamma_\eps,f)$ is Cauchy in $L^2$.
\item Tightness in the Sobolev space $\mathcal{H}^{-1-\eps}(D)$: the squared $\mathcal{H}^{-1}(D)$ norm of $\Gamma_\eps$ is equal to $\int_{D \times D} \Gamma_\eps(z) \G_D(z,w) \Gamma_\eps(w)dzdw$ and its expectation which is bounded by $\int_{D \times D}(\log |z-w|)^2dzdw < \infty$.
\end{itemize}
As each approximate field $\Gamma^n$ in the orthogonal decomposition is a finite-dimensional Gaussian, the orthogonal approximation gives a way to extend the classical finite-dimensional Gaussian shifting to the case of the GFF:

\begin{thm}[Cameron-Martin shift for the GFF]
	Let $F$ be a function belonging to $\mathcal H^1_0(D)$ and $\Gamma$ a GFF on $D$. Denote the law of $\Gamma$ by $\P$ and the law of $\Gamma+F$ by $\tilde \P$. 
	Then $\P$ and $\tilde \P$ are mutually absolutely continuous and the Radon-Nikodym derivative $d\tilde \P / d \P$ at $\Gamma$ is given by $\exp((F, \Gamma)_\nabla - \frac{1}{2}\| F \|^2_\nabla)$.
\end{thm}  
There are two key properties of the GFF that we make heavy use of:

\begin{thm}[GFF properties]\label{prop:GFFprp}
The 2D GFF satisfies the following nice properties:
\begin{enumerate}
\item (\textit{Conformal invariance}) If $\phi: D \rightarrow D'$ is a conformal map, then $\phi(\Gamma^D)$ has the law of the GFF in $D'$ (here $\phi(\Gamma^D)$ can again be given meaning via an approximation of $\Gamma^D$).
\item (\textit{Markov property}) If $B$ is any deterministic open subset of $D$, then the GFF can be written as an independent sum $\Gamma = \Gamma^B + \Gamma_B$, where $\Gamma^B$ has the law of a zero boundary GFF inside $B$ and is extended by $0$ outside, and $\Gamma_B$ is a random distribution that is harmonic inside $B$. We denote this harmonic part in $B$, that is defined pointwise and is given by the harmonic extension of the field values on $\partial B$, by $h_B$. 
\end{enumerate}
\end{thm}
Henceforth we drop $D$ and restrict ourselves mostly to the unit disk $D = \D$ for simplicity and concreteness, in which case we write $\Gamma = \Gamma^\D$. 

A simple calculation shows that for a fixed $z$, the circle-average process $\Gamma_\eps(z)$, with $\eps \leq d(z,\partial D)$, and parametrized using $-\log \eps + \log \CR(z,\partial D)$ has the law of a Brownian motion started from $\Gamma_{d(z,\partial D)}(z)$. Here by $\CR(z,\partial D)$ we denote the conformal radius of $D$ at the point $z$. 

We will also use one simple continuity estimate on the circle-average process that follows from a direct calculation using the Green's kernel, see e.g. \cite{TP}:

\begin{lemma}\label{lem:ca}
Let $\Gamma_\eps$ denote the circle-average process. Then there is a constant $C$ such that for all $\eps \leq 1$, and for all $z,w \in \D$ we have that
$$\E\left[(\Gamma_\eps(z) - \Gamma_\eps(w))^2\right] \leq C\frac{|z-w|}{\eps}.$$
\end{lemma}




\subsection{Some auxiliary results}

\subsubsection{Concentration of Gaussian measures}
We recall here two standard results from the theory of Gaussian processes, estimating the variance of the supremum of a continuous Gaussian process and its mean. Both can be found in any standard reference on Gaussian processes, for example \cite{Adler}.

\begin{thm}[Borell-TIS]\label{thm:TIS}
Let $\Gamma$ be a zero mean and almost surely finite Gaussian process on a compact metric space $T$. Denote by $S$ its supremum and let $\sigma^2 = \sup_{t \in T} \E \Gamma(t)^2$ . Then $S -\E S$ is $\sigma^2$-sub-Gaussian, i.e. for all $\lambda \in \R$
$$\E e^{\lambda (S - \E S)} \leq e^{\frac{\lambda^2\sigma^2}{2}}.$$
A similar claim also holds for the infimum.
\end{thm}

\begin{thm}[Dudley entropy bound]\label{thm:Dudley}
Let $\Gamma$ be a zero mean Gaussian process on a compact metric space $T$. Denote by $S$ its supremum and define for $s,t \in T$, the Gaussian pseudo-metric $d(t,s) = \sqrt{\E\left[(\Gamma(t) - \Gamma(s))^2\right]}$. Then denoting by $N(\eps, T, d)$ the covering number of $T$ w.r.t this pseudo-metric, we have 
$$\E S \leq 24\int_0^\infty \sqrt{\log N(\eps,T,d)}d\eps.$$
Similar bound also holds to bound the infimum from below.
\end{thm}

\subsubsection{Convergence of random measures}
For a measure $M(dz)$ on a bounded domain $D$ and a continuous bounded function $f$ denote $M(f) = \int_D f M(dz)$. A very helpful result and the basis for all the convergence arguments of GMC measures is the fact that in order to show weak convergence of the random measures $M_n$, it suffices to guarantee convergence of $M_n(f)$ for a countable number of functions $f$:
 
\begin{thm}\label{thm:conv}
Let $(M_\eps(dz))_{\eps > 0}$ be random finite non-negative measures on a closed bounded domain $\overline{D}$. Suppose that for any positive continuous function $f$ on $\overline{D}$ we have that $M_\eps(f)$ converges to some $M(f)$ in law / in probability / almost surely. Then the $M(f)$ characterize a finite non-negative measure $M(dz)$ on $\overline{D}$ and $M_\eps(dz)$ converge weakly in law / in probability / almost surely to $M(dz)$. 
\end{thm}

The theorem follows from the fact that the weak convergence of measures on $\mathcal{M}(\overline{D})$, the space of finite measures on $\overline D$, can be metrized by defining for any two finite measures $M_1, M_2 \in \mathcal{M}(\overline{D})$ $$d(M_1, M_2) = \sum_i \frac{|M_1(\phi_i) - M_2(\phi_i)|}{2^i}.$$ Here $(\phi_i)_{i \in \N}$ form a positive (Schauder) basis of the space of continuous bounded functions on $\overline{D}$ such that $||\phi_i||_\infty = 1$ for all $i \in \N$. Such a basis exists as the space of continuous bounded functions on $\overline{D}$ with uniform norm is separable, or indeed can be explicitly constructed.

See e.g. \cite{Gar} or \cite{Ber} for convergence in probability of the GMC measure, or \cite{Kb} for general statements about convergence in law for random measures. 

\section{Second moment calculation and first constructions of the GMC}\label{sec:basic}

The idea of the basic construction of a GMC measure is just to approximate $\Gamma$ using pointwise defined fields $\Gamma_\eps$ from Theorem \ref{thm:GFFapp}, then define approximate exponentials $$M_\eps(dz) = \exp(\gamma \Gamma_\eps(z) - \frac{\gamma^2}{2}\E \Gamma_\eps^2(z))dz,$$ where by $dz$ we denote Lebesgue measure, and finally show that $M_\eps(dz)$ converges to a non-trivial limit as $\eps \to 0$. In the regime $|\gamma| < \sqrt{2}$ the convergence to a non-trivial limit is relatively easy to justify via a second-moment calculation. Observe that the regularization has been chosen such that $\E M_\eps(dz)$ is precisely Lebesgue measure.

Let us point out straight away that the techniques in this section apply to all log-correlated Gaussian fields in all dimensions, e.g. to those with covariance kernel $\C(z,w) = - \log |z-w| + g(z,w)$ with $g(z,w)$ bounded and continuous. Moreover, they also apply when the underlying base measure is not Lebesgue measure, but rather some finite Radon measure of dimension $d$; in this case the the $L^2$ regime corresponds to $|\gamma| < \sqrt{d}$.

\subsection*{Second moment calculation} Let $f$ be a positive continuous function on $\overline{\D}$ and $\Gamma_\eps$ be some approximation of $\Gamma$ in Theorem \ref{thm:GFFapp}. Denote its covariance kernel by $\C_\eps(z,w)$. We have $$\E\left[M_\eps(f)^2\right] = \E \left[\int_{\D \times \D}f(z)f(w)  e^{\gamma (\Gamma_\eps(z) + \Gamma_\eps(w)) - \frac{\gamma^2}{2}(\C_\eps(z,z) - \C_\eps(w,w))}dzdw\right],$$ 
which by taking the expectation inside and evaluating exponential moments of the Gaussian $\Gamma_\eps(z) + \Gamma_\eps(w)$ is equal to
$$\int_{\D \times \D}f(z)f(w)\exp(\gamma^2 \C_\eps(z,w)) dzdw.$$
Now it is direct to check that there is $c > 0$ s.t. for all $z,w \in D$ it holds $\C_\eps(z,w) \leq -\log |z-w| + c$. Thus we can bound this by
$$C'\int_{\D \times \D}|z-w|^{-\gamma^2} dzdw,$$
which is finite if and only if $|\gamma| < \sqrt{2}$. 

Roughly this second moment argument provides the uniform integrability of the variables $\Gamma_\eps(f)$ in this regime. Hence it also gives tightness in law. 

\subsection*{Convergence for the martingale sequence}

For the orthogonal sequence of Theorem \ref{thm:GFFapp}, $M(\Gamma^n,f)$ is a positive martingale in $n \in \N$. Thus it converges almost surely to a finite limit. But the calculation above shows that for $|\gamma| < \sqrt{2}$ this martingale is also uniformly integrable and in fact in $L^2(\P)$. Hence we have convergence in $L^2(\P)$. Thus the limit is non-zero with positive probability. Finally, in fact the limit is almost surely non-zero due to a zero-one argument: the event $M(\Gamma,f) = 0$ is independent of $\Gamma^n$, for any finite $n$.

\subsection*{Convergence for mollifier approximations}
To obtain convergence for the mollifier approximations we will need to do some more work. Maybe the most direct way is to just show that the sequence $M_\eps(f)$ is in fact Cauchy in $L^2(\P)$, i.e. to show that
$$\E\left[(M_{\eps}(f)-M_{\eps'}(f))^2\right]$$
converges to zero uniformly in $\eps' \geq \eps$. This expectation is equal to
$$\int_{\D \times \D}dzdw\left(\exp(\gamma^2\C_\eps(z,w)) + \exp(\gamma^2 \C_{\eps'}(z,w) - 2\exp(\gamma^2\E \Gamma_\eps(z) \Gamma_{\eps'}(w))\right).$$
For the mollifier and circle-average approximations, this can be bounded by a similar, but slightly more careful calculation than the one above. The heart of the calculation is to notice that $\C_\eps(z,w) = \C_{\eps'}(z,w)$ for $|z-w| \geq 2\eps'$. On the other hand the integral over the near-diagonal part, $|z-w| \leq 2\eps'$, vanishes with $\eps' \to 0$.

\subsection*{Uniqueness using uniform integrability}

Let $\Gamma^n$ be the orthogonal sequence, and $\Gamma_\eps$ be some other approximation sequence as above. We saw that $M(\Gamma_n,f) \to M(f)$ almost surely and that $M(\Gamma_\eps,f)$ converges in $L^1(\P)$. But why should these limits be the same? 

In fact, this follow from just three simple observations:

\begin{enumerate}
\item For the orthogonal sequence we can write $M(\Gamma^n,f) = \E\left[M(f) | \Gamma^n\right]$.
\item An explicit calculation shows that $\E \left[M_{\eps}(f) | \Gamma^n\right]$ converges almost surely to $M(\Gamma^n, f)$ as $\eps \to 0$. 
\item Convergence of $M_{\eps}(f) \to \tilde M(f)$ in $L^1(\P)$ shows that $\E \left[M_{\eps}(f) | \Gamma^n\right]$ also converges in $L^1(\P)$ to $\E \left[\tilde M(f) | \Gamma^n\right]$. 
\end{enumerate}

We deduce that the the limit $\tilde M(f)$ agrees with $M(f)$. 

\subsubsection*{Basic result in the $L^2$-regime}

Combining the calculations above together with Theorem \ref{thm:conv} we have obtained the following basic result:

\begin{prop}\label{prop:basic}
Let $|\gamma| < \sqrt{2}$ and consider the approximating fields $\Gamma_\eps$ as in Theorem \ref{thm:GFFapp}. Then for all these approximate measures $$M_\eps(dz) = M^\gamma(\Gamma_\eps, dz) = e^{\gamma \Gamma_\eps - \frac{\gamma^2}{2}\C_\eps(z,z)}dz$$
converge weakly in probability to a unique non-trivial measure $M(dz)$.
\end{prop}

Moreover, notice that the same second-moment calculation gives the following corollary, which, as we will also see later is not yet sharp:

\begin{cor}[Dimension of the support]\label{cor:supp}
Let $\gamma \in (-\sqrt{2}, \sqrt{2})$. Then for any $a < 2-\gamma^2$, we have that $\E \int_D \frac{M^\gamma(dz)M^\gamma(dw)}{|z-w|^a} < C$. In particular Frostman's lemma implies that almost surely the Hausdorff dimension of the support of the measure is at least $2-\gamma^2$.
\end{cor}

The fact that we have several possible approximations, all giving the same measure allows us to prove properties of the measure $M^\gamma(\Gamma,dz)$ by choosing a convenient approximation for which the property is easier to prove. The following result in its exact form is inherent to the GFF:

\begin{lemma}[Change of co-ordinates]\label{lem:coc}
Let $\phi: \tilde D \to D$ be a conformal map. Then the push-forward of $M^\gamma(\Gamma^D, dz)$ from $D$ along $\phi^{-1}$ to $\tilde D$ has the same law as the measure $M^\gamma(\Gamma^{\tilde D} + \frac{2}{\gamma} \log |\phi'|)$.
\end{lemma}

\begin{proof}
Consider the orthogonal approximation of the GFF and notice that the orthonormal Dirichlet eigenfunctions of $\Delta$ on $D$ map to the orthonormal Dirichlet eigenfunctions on $D'$. Thus denoting $\Gamma = \Gamma^D$ and $\tilde \Gamma = \Gamma^{\tilde D}$, we have that $\phi(\Gamma^n)$ has the law of $\phi(\tilde \Gamma^n)$. Hence the pushforward of $M(\Gamma^n, dz)$ has the law of $M(\tilde \Gamma^n, |\phi'(z)|^2dz)$. But $M(\tilde \Gamma^n, |\phi'(z)|^2dz) = M(\tilde \Gamma^n + \frac{2}{\gamma} \log \phi', dz)$ and thus the claim follows.
\end{proof}

\subsection*{Convergence of the process in $\gamma$}

In fact, it comes out that the convergence can easily be lifted to the convergence of $M^\gamma(\Gamma_\eps,f)$ seen as a continuous function of $\gamma$:

\begin{prop}
Let $f$ be continuous and bounded on $\overline{\D}$. Then process $M^\gamma(\Gamma,f)$ is H\"older continuous in $\gamma$ over any compact sub-interval $I$ of $(-\sqrt{2},\sqrt{2})$. Moreover, if $\Gamma_\eps \to \Gamma$ as in Proposition \ref{prop:basic}, then $(M^\gamma(\Gamma_\eps, f))_{\gamma \in I} \to (M^\gamma(\Gamma,f))_{\gamma \in I}$ in probability in the space of continuous functions on $I$ with the uniform norm.
\end{prop}

\begin{proof}
There are probably several ways to see this. A direct way is to just observe that there is some $C >0$ such that for all $\eps > 0$ and for all $\gamma, \tilde \gamma \in I$, we have
\begin{equation}\label{eq:Hldr}
\E \left[(M^\gamma(\Gamma_\eps, f) - M^{\tilde \gamma}(\Gamma_\eps, f))^2\right] \leq C|\gamma -\tilde \gamma|^2.
\end{equation}
The calculation is very similar to the one above, but let us sketch it here for completeness. Suppose first $\eps >0$ and let $\tilde \gamma = \gamma + \delta$. Writing $e_{\gamma_1, \gamma_2}(z,w) := e^{\gamma_1 \Gamma_\eps(z) + \gamma_2 \Gamma_\eps{w} - \frac{\gamma_1^2}{2}\E \Gamma_\eps(z)^2 - \frac{\gamma_2^2}{2}\E \Gamma_\eps(w)^2}$, we get that the LHS of \eqref{eq:Hldr} equals
$$\int_{D \times D} dzdwf(z)f(w)\E\left( e_{\gamma, \gamma}(z,w) - 2e_{\gamma, \gamma + \delta}(z,w) + e_{\gamma + \delta, \gamma + \delta}(z,w)\right).$$
Now $\E e_{\gamma_1, \gamma_2}(z,w) = e^{\gamma_1 \gamma_2 \C_\eps(z,w)}$ and thus we can rewrite this as 
$$\int_{D \times D} dzdwf(z)f(w) e^{(\gamma+\delta)^2\C_\eps(z,w)}\left(1+e^{-\delta(2\gamma +\delta) \C_\eps(z,w)} - 2e^{-\delta(\gamma+\delta) \C_\eps(z,w)}\right).$$
Using $1- x \leq e^{-x} \leq 1 - x + x^2/2$ we further bound this from above by
$$c\delta^2 \int_{D \times D} dzdwf(z)f(w) e^{(\gamma+\delta)^2\C_\eps(z,w)}\C_\eps(z,w)^2.$$
But as long as $\delta + \gamma \in (-\sqrt{2}, \sqrt{2})$ the integral is uniformly bounded.  Thus we obtain the estimate \eqref{eq:Hldr} for $\eps >0$ with $C$ uniform. By dominated convergence it also holds in the limit (i.e. with $\eps = 0$), giving H\"older continuity in $\gamma$ for (a modification of) $M^\gamma(\Gamma,f)$ using the Kolmogorov criteria. Moreover, \eqref{eq:Hldr} also gives tightness in H\"older continuous functions with H\"older exponent less than $1/2$ and thus the convergence.
\end{proof}
 
The statement of the proposition is certainly not optimal. In fact also any derivative of $M^\gamma(\Gamma_\eps,f)$ in $\gamma$ converges as a process. One way to see this is to use again a direct calculation. Let us sketch it for the first derivative. Denote $D^\gamma(\Gamma_\eps, dz) = \frac{d}{d\gamma} M^\gamma(\Gamma_\eps,dz)$. We have 
$$D^\gamma(\Gamma_\eps, dz) = (\Gamma_\eps(z) - \gamma \E \Gamma_\eps^2(z))e^{\gamma \Gamma_\eps(z) - \frac{\gamma^2}{2}\E \Gamma_\eps(z)^2}dz.$$
A Gaussian calculation, interpreting $e^{\gamma (\Gamma_\eps(z)+\Gamma_\eps(w))) - \frac{\gamma^2}{2}\E (\Gamma_\eps(z)+\Gamma_\eps(z))^2}$ as a Cameron-Martin weight, shows that for $|\gamma| < \sqrt{2}$, 
$$\E\left[ D^\gamma(\Gamma_\eps, f)^2\right] = \int_{D \times D} dzdwf(z)f(w)\gamma^2\C_\eps^2(z,w)e^{\gamma^2 \C_\eps(z,w)}.$$
Using $\C_\eps(z,w) \leq -\log |z-w|$, we see that the second moment is bounded uniformly in $\eps$. As for the orthogonal sequence $\Gamma^n$, we see that $D^\gamma(\Gamma^n,f)$ is still a martingale. It is no longer positive, but as it is uniformly integrable, it still converges almost surely and in $L^1$. Notice that in this case $D^\gamma(\Gamma)$ is no longer a measure, but lives in the space of distributions.

We can further write $\E \left[(D^\gamma(\Gamma_\eps, f) - D^{\tilde \gamma}(\Gamma_\eps, f))^2\right]$ as
$$\int_{D \times D}dzdwf(z)f(w)\gamma(\gamma+\delta)\C_\eps^2(z,w)e^{(\gamma+\delta)^2\C_\eps(z,w)}\left(-2e^{-\gamma(\gamma + \delta) \C_\eps^2(z,w)} + (1-\frac{\delta}{\gamma + \delta})e^{-\delta(2\gamma + \delta) \C_\eps^2(z,w)} + (1+\frac{\delta}{\gamma})\right).$$
Bounding the exponential as above, we again conclude that there is some $C >0$ such that for all $\eps$ sufficiently small and for all $\gamma, \tilde \gamma \in I$, we have
\begin{equation}\label{eq:HldrD}
\E \left[(D^\gamma(\Gamma_\eps, f) - D^{\tilde \gamma}(\Gamma_\eps, f))^2\right] \leq C|\gamma -\tilde \gamma|^2.
\end{equation}
Hence the derivative as a process in $\gamma$ also converges in the space of H\"older regular functions on any compact subinterval of $(-\sqrt{2}, \sqrt{2})$. In particular we see that the derivative of $M^\gamma(\Gamma,f)$ exists in this region and is equal to $D^\gamma(\Gamma,f)$.

\begin{rem} There are a few remarks to be made:
\begin{itemize} 
\item When one considers a complex exponent $\gamma$, then one can prove analyticity of $M^\gamma(\Gamma,f)$ in $\gamma$ even with slightly fewer calculations: analyticity at the approximate level, uniform convergence and Cauchy integral formula will give analyticity at the limit.
\item Using the techniques of Section \ref{sec:RM}, these calculations extend to the whole subcritical range. 
\item In fact one define a continuous process in the space of finite measures. This follows by the fact that atomless measures are for example characterized by testing them against all squares with dyadic co-ordinates.
\end{itemize}
\end{rem}

Finally, there is also a rather cute corollary:

\begin{cor}
Let $I$ be some open interval around $0$. Then in the coupling $\left(\Gamma,(M^{\gamma}(\Gamma,dz))_{\gamma \in I}\right)$, where $(M^{\gamma}(\Gamma,dz))_{\gamma \in I}$ is defined via an approximation as above, we have that $\Gamma$ is measurable w.r.t. $(M^{\gamma}(\Gamma),dz)_{\gamma \in I}$.
\end{cor}

\begin{proof}
To see this first notice that for any fixed continuous $f$, $D^0(\Gamma_\eps,f) = (\Gamma_\eps,f)$ and by convergence in probability $D^0(\Gamma,f) = (\Gamma,f)$. Now as the derivative process is continuous in the limit, it follows that $D^{\gamma}(\Gamma,f) \to D^0(\Gamma,f)$ almost surely as $\gamma \to 0$. Finally, the derivative process $D^{\gamma}(\Gamma,f)$ is measurable w.r.t $M^\gamma(\Gamma,f)$. As this all holds for a countable number of $f$, and we have tightness in the space of distributions, the corollary follows.
\end{proof}

In Section \ref{sec:meas} we will prove a stronger result saying that the underlying GFF is measurable w.r.t. the GMC measure for a fixed $\gamma$.

\subsection*{References to the literature}

The second moment argument goes back to \cite{HK}, where the authors use $\exp(\gamma \Gamma)$ as an interaction term for free fields and it also appears in Kahane's article \cite{Kah} for general dimensions and general base measures. The idea to use the interplay between martingale sequences and circle-average approximations is used in \cite{DS}, it is further developed in \cite{Ber} to prove uniqueness for mollifier approximations, and is finally used in \cite{APS} to prove convergence to the GMC measure for certain cascade-like approximations. 

The convergence of the whole process in 1D goes back at least to \cite{Bar}, and is also used in \cite{RCW}. In both cases the authors in fact work with complex parameter $\gamma$ and prove analyticity in the disk around zero using Cauchy's integral formula. The measurability of the GFF from the GMC measures defined around $0$ can be probably considered folklore.

\section{Rooted measures, thick points and typical points}\label{sec:RM}

Let $\gamma < \sqrt{2}$ be fixed. Suppose we know how to construct the measure $M^\gamma(\Gamma,dz)$ and you pick a ``typical point'' according to this measure. How does the measure look around this point? 

A neat way to study this question is to consider the rooted (also known as Peyri\`ere) measures, i.e. measures together with a point randomly sampled from the measure. In the case of the GMC measure there is a very explicit definition and description of these measures in terms of the underlying field.

Indeed, let $\Gamma_\eps$ be some approximation as before, which for concreteness we could fix to be the circle-average approximation. Then an approximate rooted measure is the probability law $\P^*_\eps(dz,d\Gamma^*_\eps)$ on $D\times \mathcal{D}(D)$, where $\mathcal{D}(D)$ is the space of distributions on $D$. 

Its law is given by $$\P^*_\eps(dz,d\Gamma^*_\eps) = \frac{1}{\Leb (D)}e^{\gamma \Gamma^*_\eps - \frac{\gamma^2}{2}\E {\Gamma^*_\eps}^2}dzd\Gamma^*_\eps.$$ Here $dz$ denotes Lebesgue measure on $D$ and $d\Gamma^*_\eps = d\Gamma^D_\eps$ is the law of the GFF approximation. Observe that
\begin{itemize}
\item The marginal law on $D$ is proportional to Lebesgue measure;
\item the marginal law on $\mathcal{D}(D)$ is proportional to $M_\eps(D)d\Gamma_\eps$
\item the conditional (probability) law on $D$, conditioned on $\Gamma^*_\eps$ is equal to $\frac{1}{M_\eps(D)}M_\eps(dz)$;
\item the conditional (probability) law on $\mathcal{D}(D)$, conditioned on $z$ is by the Cameron-Martin theorem equal to $\Gamma_\eps(w) + \gamma \C_\eps(z,w).$
\end{itemize}
Thus, we can sample an instance of the approximate rooted measure by:
\begin{enumerate}
\item first sampling $\Gamma^*_\eps$ from $M_\eps(D)d\Gamma^*_\eps$, then sampling a point according to the probability measure $M_\eps(dz)/M_\eps(D)$.
\end{enumerate}
We see from this description that as $M_\eps(D)$ converges to $M(D)$ in probability, $\P^*_\eps$ converges in law to a rooted measure $\P^*$.  

There is however, also another way of sampling an instance of the approximate rooted measure:
\begin{enumerate}
\setcounter{enumi}{1}
\item we first sample $z$ from Lebesgue measure and then an element from $\mathcal{D}(D)$ with the law of $\Gamma_\eps(w) + \gamma \C_\eps(z,w).$
\end{enumerate}
Combining the descriptions (1) and (2) of the approximate rooted measure, we get the following lemma: 
\begin{lemma}[Rooted measure]\label{lem:rm}
For any $\eps \geq 0$, and any continuous bounded functional $F$ on $\overline{D} \times \mathcal{D}(D)$, we have
$$\E \int_D F(z,\Gamma_\eps) M_\eps(dz) = \E \int_D F(z, \Gamma_\eps + \gamma \C_\eps(z,\cdot))dz,$$
where the expectation is w.r.t. $\Gamma$. Moreover, if $M_\eps(dz)$ converges in the space of finite measures on $D$ as $\eps \to 0$, the same holds with $\eps = 0$ and $\C_0(z,w) = \G_D(z,w)$. 
\end{lemma}

Notice that in the other direction, if there is some random measure $M(\Gamma,dz)$ such that for all continuous functionals $F$ on $D\times \mathcal{D}(D)$
$$\E \int_D F(z, \Gamma) M(\Gamma, dz) = \E \int_D F(z, \Gamma + \gamma \G_D(z,\cdot))dz,$$
then this uniquely characterizes the rooted measure, but also the measure $M$: indeed, the right hand side does not depend on $M(\Gamma,dz)$ at all. We will come back to this in the beginning of Section \ref{sec:intr}, where we state it as Proposition \ref{prop:chr1}. For now we just observe the following: if we have tightness in law for $M(\Gamma_\eps,dz)$, then this characterization shows that all subsequential limits have to agree, implying convergence in law. 

Finally let us return to the question of the beginning of the section. As almost surely $0 < M(D) < \infty$, the law of $\Gamma^*$ is absolutely continuous w.r.t. the law of $\Gamma$. Thus the almost sure properties of the field around a ``typical point'' sampled from $M(dz)$ can be studied by looking at the conditional law of $\P^*$ on $d\Gamma^*$, conditioned on $z$. By the above, we know that this is given by the law of $\Gamma + \gamma \G_D(z, \cdot)$. 

We know from Section \ref{sec:prelim} that around a Lebesgue-typical point the circle-average $\Gamma_\eps(z)$ parametrized using $-\log \eps + \log \CR(z,\partial D)$ has the law of a standard Brownian motion. It now follows that under $\P^*$, conditioned on $z$, the circle-average has the law of a Brownian motion with drift $- \gamma \log \eps$. In particular, whereas for a Lebesgue-typical point $\frac{\Gamma_\eps(z)}{- \log \eps} \to 0$, under $\P^*$ almost surely $\frac{\Gamma^*_\eps(z)}{- \log \eps} \to \gamma$ as $\eps \to 0$. Such points are called $\gamma-$thick points, and more generally if the above limit is equal to $a$, they are called $a-$thick points. We thus conclude that:

\begin{claim}
Almost surely a point picked according to the GMC measure $M^\gamma(\Gamma,dz)$ is a $\gamma-$thick points.
\end{claim}

\subsection{Removing thick points to obtain uniform integrability of $M^\gamma(f)$ and dimension of the support}\label{sec:thickpoints}

A very neat way to deal with the blow-up of the second moment of $M^\gamma(D)$ for $\gamma \geq \sqrt{2}$ is to understand that 
\begin{itemize}
\item this blow-up stems from the contribution of points that are strictly more than $\gamma-$thick and that, 
\item on the other hand, these points do not contribute to the measure $M^\gamma(z)$ in the $L^1$ sense. 
\end{itemize}
More precisely we define:
\begin{defn}
For $\eps < \eps'$ and $a \in \R$, we say that $z$ is at least $a-$thick on $(\eps,\eps')$ if $\frac{\Gamma_\nu(z)}{-\log \nu} > a$ for some $\nu \in (\eps,\eps')$ and denote this event by $B^a_{(\eps, \eps')}(z)$. Denote the complementary event for $z$, i.e. that throughout $(\eps,\eps')$, we have $\frac{\Gamma_\nu(z)}{-\log \nu} \leq a$ by $G^a_{(\eps, \eps')}(z)$. 
\end{defn}

Then the two following lemmas give the uniform integrability of $M^\gamma(f)$. The first says that we can neglect at least $(\gamma+\delta)$-thick points in $L^1$ calculations, and the second says that without such points the second moments are finite.

\begin{lemma}[Neglecting too thick points]\label{lem:1pt}
Let $\gamma \in (-2,2)$. Suppose $f$ is non-negative, continuous and supported on $R\overline{\D}$ for some $R < 1$ and let $\eps' < 1-R, \delta > 0$. Then 
$$\E \int_\D f(z)\I{B^{\gamma+\delta}_{(\eps, \eps')}(z)}e^{\gamma \Gamma_\eps(z) - \frac{\gamma^2}{2}\E \Gamma_\eps^2(z)}dz = o_{\eps',\delta}(1)$$
uniformly over all $\eps < \eps'$.
\end{lemma}

This lemma follows easily from the Cameron-Martin theorem and a Gaussian calculation. Indeed, when taking $f = 1$ for simplicity, the LHS is up to a constant equal to $\P^*(B^{\gamma+\delta}_{(\eps, \eps')}(z)).$
But this is equal to the probability that a Brownian motion with drift $\gamma$ stays above $\gamma + \delta$ throughout the interval $[-\log \eps', -\log \eps]$, which is of order $o_{\eps'}(1)$ for any fixed $\delta > 0$.

\begin{lemma}[Thinned out second moments exit]\label{lem:2ndme}
Let $\gamma \in (-2,2)$. Suppose $f$ is non-negative, continuous and supported on $R\overline{\D}$ for some $R < 1$ and let $\eps' < 1-R$. Then for $\delta$ small enough depending only on $\gamma$, there is some $C > 0$ such that
$$\E \left[(\int_\D f(z)\I{G^{\gamma+\delta}_{(\eps, \eps')}(z)}e^{\gamma \Gamma_\eps(z) - \frac{\gamma^2}{2}\E \Gamma_\eps^2(z)}dz)^2\right] < C$$
for all $\eps < \eps'$.
\end{lemma}

Uniform integrability of $M_\eps(f)$ then follows: indeed, first choose $\delta >0$ such that the thinned-out second moment is finite. Then given $\nu > 0$, we can choose by Lemma \ref{lem:1pt} an $\eps'$ such that $\E\left[M_\eps(f\I{B^{\gamma+\delta}_{(\eps, \eps')}})\right] \leq \nu / 2$. Finally, from Lemma \ref{lem:2ndme} we can choose $K > 0$ such that for $X = M_\eps(f\I{G^{\gamma+\delta}_{(\eps, \eps')}})$, we have $\E\left[X\I{X\geq K} \right]\leq \nu/2$. 

\begin{proof}[Proof of Lemma \ref{lem:2ndme}]
Write the LHS of the statement as
$$\E \left[(\int_{\D \times \D}f(z)f(w)dzdw \I{G^{\gamma+\delta}_{(\eps, \eps')}(z)}\I{G^{\gamma+\delta}_{(\eps, \eps')}(w)}e^{\gamma (\Gamma_\eps(z)+\Gamma_\eps(w) - \frac{\gamma^2}{2}(\E \Gamma_\eps^2(z)+\E \Gamma_\eps^2(w))}dz)\right].$$
Bringing the expectation inside the integral and then interpreting $e^{\gamma (\Gamma_\eps(z)+\Gamma_\eps(w))) - \frac{\gamma^2}{2}\E (\Gamma_\eps(z)+\Gamma_\eps(z))^2}$ as a Cameron-Martin weight, we can rewrite this expectation as
$$\int_{\D \times \D}f(z)f(w)dzdw e^{\gamma^2 \C_\eps(z,w)}\tilde \P(G^{\gamma+\delta}_{(\eps, \eps')}(z) \cap G^{\gamma+\delta}_{(\eps, \eps')}(w)).$$
Here under $\tilde \P$, for $u \in \D$, $\Gamma_\eps(u)$ has the law of $\Gamma_\eps(u) + \gamma \C_\eps(z,u) + \gamma \C_\eps(w,u)$.
We bound the integral over $|z-w| > \eps'/2$ just by $O(1/\eps')$. To bound the integral over $|z-w| \leq \eps'/2$ we first bound trivially
$$\tilde \P(G^{\gamma+\delta}_{(\eps, \eps')}(z) \cap G^{\gamma+\delta}_{(\eps, \eps')}(w)) \leq \tilde \P(G^{\gamma+\delta}_{(\eps, \eps')}(z)).$$
Let now $\hat \epsilon = |z-w|\vee \eps$. We bound $\tilde \P[G^\gamma_{(\eps, \eps')}(z)]$ further by \[\P\left[\Gamma_{\hat \eps}(z) + \gamma(\C_{\hat \eps}(z,z) + \C_{\hat \eps}(z,w)) \leq -(\gamma+\delta)\log \hat \eps\right].\]
Now as $\C_{\hat \eps}(z,z) + \C_{\hat \eps}(z,w) \geq -2\gamma\log \hat \eps -c$ for some constant $c$ uniform over $z,w \in R\D$, we see that this probability is bounded by the probability that a Gaussian of variance $-\log \hat \eps$ and mean $-2\gamma \log \hat \eps$ is smaller than $-(\gamma+\delta)\log\hat \eps$. This probability is of order ${\hat \eps}^{(\gamma - \delta)^2/2}$. Thus
$$e^{\gamma^2 \C_\eps(z,w)}\tilde \P(G^{\gamma+\delta}_{(\eps, \eps')}(z) \cap G^{\gamma+\delta}_{(\eps, \eps')}(w) ) \lesssim (\hat \eps)^{-\gamma^2 + (\gamma - \delta)^2/2}.$$
 And thus the integral over $ |z-w| \leq \eps'/2$ can be bounded by $$C'\int_{|z-w| \leq \eps'/2} \frac{1}{|z-w|^{\gamma^2 - (\gamma - \delta)^2/2}}.$$ But for $\delta$ small enough the exponent is smaller than $2$, the integral is of order $o_{\eps'}(1)$ and thus we obtain the claim. 
\end{proof}

To prove actual convergence of the measure for circle-average or mollifier approximations one would, as in Section \ref{sec:basic}, extend the second-moment calculations to prove that the thinned out version of $M(\Gamma_\eps,f)$ is Cauchy in $L^2$. Together with Lemma \ref{lem:1pt} this gives that $M(\Gamma_\eps,f)$ is Cauchy in $L^1$. 

For the orthogonal sequence $\Gamma^n$ the convergence is again easier. From de la Vall\'ee- Poussin theorem, and the fact that $\E \left[M_{\eps}(f) | \Gamma^n\right]$ converges to $M(\Gamma^n, f)$ as $\eps \to 0$ one can deduce uniform integrability also for the orthogonal sequence. (See end of Section \ref{sec:UI1} for a very similar argument transferring existence of positive moments from circle-average approximations to orthogonal approximation.) As $M(\Gamma^n,f)$ is a positive martingale, the convergence follows. 

As in the above lemma we work for convenience in $R\overline{\D}$, this a priori only gives convergence on the space of measures on $R\overline{\D}$ with $R < 1$. However, as $R < 1$ is arbitrary and $\E M(\Gamma,R\overline{\D}) = \Leb(R\overline{\D}) \to \Leb(\D)$ as $R \to 1$, one can deduce from monotone convergence that $\E M(\Gamma,\overline{\D}) = 1$. This allows one to extend the convergence to measures on $\overline{\D}$. Let us state this all as a theorem:

\begin{thm}\label{thm:basic}
Let $|\gamma| < 2$ and consider the approximating fields $\Gamma_\eps$ as in Theorem \ref{thm:GFFapp}. Then for all these approximate measures $$M_\eps(dz) = M^\gamma(\Gamma_\eps, dz) = e^{\gamma \Gamma_\eps - \frac{\gamma^2}{2}\C_\eps(z,z)}dz$$
converge as measures on $\overline{\D}$ weakly in probability to a unique non-trivial measure $M(dz)$.
\end{thm}

The calculation in the proof of Lemma \ref{lem:2ndme} also says that for any $a < 2-\gamma^2/2$, we have that $\E \int_\D \I{G^{\gamma+\delta}_{(0, \eps')}(z)}\I{G^{\gamma+\delta}_{(0, \eps')}(w)}\frac{M^\gamma(dz)M^\gamma(dw)}{|z-w|^a} < C$. In particular Frostman's lemma then helps to turn the Corollary \ref{cor:supp} to a sharp estimate:
\begin{cor}\label{cor:dimsharp}
Let $\gamma \in (-2, 2)$. Then  almost surely the Hausdorff dimension of the support of the measure is at least $2-\gamma^2/2$.
\end{cor}

Combining this with the fact that $M^\gamma(\Gamma,dz)$ is supported on $\gamma-$thick points gives us a lower bound on the Hausdorff dimension of $\gamma-$thick points. An upper bound can be obtained by a classical 1-point estimate.

\subsection*{References to the literature}

The idea of rooted measures goes back to Peyri\`ere \cite{Per} in the case of multiplicative cascades, and is used already in \cite{Kah} for the GMC measures. The observation that rooted measures can be used characterise the GMC measure seems to go back at to at least \cite{Sato}, but was brought newly to attention by the generalisations in \cite{Sha}.

The strategy to treat the $L^1$ setting by removing thick points seems to be first used in \cite{DS}. This idea is then greatly distilled, simplified and generalized in \cite{Ber} to give very elementary and clean construction of the GMC measures in the case of log-correlated fields in all dimensions and for general base measures. 

The upper bound on the Hausdorff dimension of the $\gamma-$thick points of the 2D Gaussian free field is calculated in \cite{TP}, where the lower bound is derived from an usual second moment calculation.

\section{An intrinsic definition of the GMC measure}\label{sec:intr}
Let us now come back to the remark made after Lemma \ref{lem:rm}, and state the first characterization of the GMC measure. For clarity, let us denote by $\E_\Gamma$ the expectation w.r.t. just the GFF $\Gamma$.

\begin{prop}[Characterization of GMC as Radon-Nikodym derivative]\label{prop:chr1}
Let $\Gamma$ be a GFF on $D$. For $-2 < \gamma < 2$, there is a unique random measure $M(\Gamma,dz)=M^\gamma(\Gamma,dz)$ such that for all continuous functionals $F$ on $D \times \mathcal{D}(D)$, we have
\begin{equation}\label{eq:char}
\E_\Gamma \int_D F(z,\Gamma) M(\Gamma, dz) = \E_\Gamma \int_D F(z, \Gamma + \gamma \G_D(z,\cdot))dz.
\end{equation}
\end{prop}

The uniqueness follows from \eqref{eq:char}: indeed, for any $F$, the RHS side is just some well-defined number, and the LHS over all $F$ characterizes uniquely a probability measure on $D \times \mathcal{D}(D)$. Thus, by disintegration theorem, it also uniquely characterizes the probability measure $\frac{1}{M(\Gamma,D=}M(\Gamma, dz)$ on $D$. Finally, $M(\Gamma, D)$ is characterized by taking $F(z, \Gamma) = F(\Gamma)$.

The existence was given by our construction of rooted measures in Section \ref{sec:RM}. Notice that only now it really makes sense to talk about ``the'' GMC measure of the 2D GFF. It also directly implies the uniqueness given in Proposition \ref{prop:basic} and in Theorem \ref{thm:basic}. 

In fact there is also even a more natural characterization of the GMC as the exponential function:

\begin{prop}[Characterization of GMC as the exponential]\label{prop:chr2}
Consider a GFF $\Gamma$ on $D$. Then for $-2 < \gamma < 2$, there is a unique random measure $M(\Gamma,dz)=M^\gamma(\Gamma,dz)$ such that $\E_\Gamma M(\Gamma,dz )$ is equal to Lebesgue measure and for all functions $f \in \mathcal{H}^1_0$, we have almost surely that 
\begin{equation}\label{eq:char2}
M(\Gamma + f, dz) = e^{\gamma f(z)}M(\Gamma,dz).
\end{equation}
\end{prop}

\begin{proof}
As property \eqref{eq:char2} holds for each approximate measure $M(\Gamma^n,dz)$, the existence follows from the construction of the GMC measure via orthogonal approximations. To deduce uniqueness, it suffices to prove that this second characterization implies \eqref{eq:char}. 

So suppose we have a GFF $\Gamma$ and a random measure $M(\Gamma,dz)$ satisfying \eqref{eq:char2}. Consider again a rooted measure $\Q^*$ (this time not normalized to have unit mass) on $D\times \mathcal{D}(D)$ corresponding to $M(\Gamma,dz)$ and characterised by requiring for any bounded continuous functional $F$ on $D \times \mathcal{D}(D)$ 
$$\E_{\Q^*} F(z,\Gamma) = \E_\Gamma \int F(z,\Gamma)M(\Gamma,dz).$$
Observe that \eqref{eq:char} is equivalent to saying that under $\Q^*$ conditionally on $z$, $\Gamma$ has the law of $\tilde \Gamma + \gamma \G_D(z,\cdot)$, where $\tilde \Gamma$ is a GFF. To obtain the claim, it thus suffices to determine the conditional law on $\Gamma$ under $\Q^*$ conditionally on $z$. In this respect, recall that the Cameron-Martin space of $\Gamma$ is $\mathcal{H}^1_0$ and that its inner product is denoted by $(f,f)_\nabla$. Let $f \in \mathcal{H}^1_0$ and $g:D \to \R$ be some auxiliary positive and bounded continuous function on $D$. From the definition of $\Q^*$
$$\E_{\Q^*} \left[g(z) e^{(\Gamma, f)_\nabla}\right] = \E_\Gamma \left[\int_D g(z)e^{(\Gamma, f)_\nabla} M(\Gamma, dz)\right].$$ 
By Fubini this equals further
$$\int_D g(z) \E_\Gamma \left[e^{(\Gamma, f)_\nabla} M(\Gamma, dz)\right].$$
From the Cameron-Martin theorem, we have:
$$\int_D g(z) \E_\Gamma \left[e^{(\Gamma, f)_\nabla} M(\Gamma, dz)\right] = \int_D g(z) \E_\Gamma\left[e^{\frac{1}{2} (f, f)_\nabla} M(\Gamma+f, dz)\right].$$
From Equation \eqref{eq:char2}, this equals $\int_D g(z) \E_\Gamma \left[ e^{\frac{1}{2} (f, f)_\nabla} e^{\gamma f(z)}  M(\Gamma, dz)\right]$, which can be rewritten as $$\int_D g(z) e^{\frac{1}{2} (f, f)_\nabla +\gamma (\G_D(z,\cdot), f)_\nabla}\E_\Gamma\left[M(\Gamma, dz)\right].$$
But $\E_\Gamma M(\Gamma,dz)$ is Lebesgue measure. Hence, recognizing $$e^{\frac{1}{2} (f, f)_\nabla +\gamma (\G_D(z,\cdot), f)_\nabla}=\E_{\tilde \Gamma} \left[e^{(\tilde \Gamma+\gamma \G(z, \cdot),f)_\nabla}\right]$$ with $\tilde \Gamma$ a GFF, we can write this as
$$\int_D dzg(z) \E_{\tilde \Gamma} \left[e^{(\tilde \Gamma+\gamma \G(z, \cdot),f)_\nabla}\right]$$
which characterizes the conditional law of $(\Gamma,f)_\nabla$ under $\Q^*$ conditionally on $z$ to be that of $(\tilde \Gamma+\gamma \G(z, \cdot),f)_\nabla$. As this holds for all $f \in \mathcal{H}^1_0$, we deduce that $\Gamma$ under $\Q^*$ conditionally on $z$ has the law of $\tilde \Gamma + \gamma \G_D(z,\cdot)$, where $\tilde \Gamma$ is a GFF.
\end{proof}
\subsection{GMC on top of GMC and another route to the $L^1$ regime}

We now know what the GMC measure is, and thus can really verify whether a construction gives the GMC measure. As an example, we will show how the knowledge about the dimension of the Liouville measure for $\gamma < \sqrt{2}$ allows us to give yet another construction of the measure in the regime $\gamma \geq \sqrt{2}$. 

The idea is to construct an independent GMC measure with parameter $a$ on top of  $M^\gamma(\Gamma, dz)$ with $\gamma < \sqrt{2}$. In other words, let $\tilde \Gamma$ be an independent GFF on $D$. We then define $$\widetilde M_n = M^a(\tilde \Gamma^n, M^\gamma(\Gamma,dz)),$$
where $\tilde \Gamma^n$ is the orthogonal approximation from Theorem \ref{thm:GFFapp}. 
From Corollary \ref{cor:dimsharp} and from Frostman's lemma we know that for $\gamma \in (-\sqrt{2}, \sqrt{2})$ (in fact for $\gamma \in (-2,2)$) and any $a < 2-\gamma^2/2$ almost surely $$\int_D \frac{M^\gamma(dz)M^\gamma(dw)}{|z-w|^a} < C.$$ 
Hence the same second moment calculation as in Section \ref{sec:basic} shows that if $a^2 < 2 -\frac{\gamma^2}{2}$, then for any a.s. instance of $\Gamma$, we have that $\widetilde M_n$ converges in $L^2(d\tilde \Gamma)$ to some resulting measure $\widetilde M(dz)$. We claim that it is a GMC measure:
\begin{prop}
$\widetilde M(dz)$ has the law of the GMC measure with parameter $\sqrt{\gamma^2 + a^2}$ w.r.t. the field $\widehat \Gamma := \frac{1}{\sqrt{a^2 + \gamma^2}} (\gamma \Gamma +a \tilde \Gamma)$. 
\end{prop}

\begin{proof}
The idea is to use Proposition \ref{prop:chr2} as a criteria. To do this, we need to show that the measure $\widetilde M(dz)$ measurable w.r.t. $\widehat \Gamma$ (by construction it is measurable only w.r.t. the pair $\Gamma, \tilde \Gamma)$) and check the condition \ref{eq:char2} of that Proposition.

 For the first part, let $f$ be any continuous bounded function on $D$. By the convergence of $\widetilde M_n(f)$ to $\widetilde M(f)$ in $L^2(d\tilde \Gamma)$ (when conditioning on $\Gamma$), and of $M^\gamma(\Gamma_n,dz)$ to $M^\gamma(\Gamma,dz)$ in $L^2(d\Gamma)$ we also have that $\E\left[\int_D f(z) M^a(\tilde \Gamma_m, M^\gamma(\Gamma_k, dz)) | \widehat \Gamma_n\right]$ converges to $\E\left[\widetilde M(f) | \widehat \Gamma_n\right]$ as first $k \to \infty$ and then $m \to \infty$ in $L^2((d\Gamma, \tilde d\Gamma))$.

But now for any $m,k > n$ we have that $$\E\left[\int_D f(z) M^a(\tilde \Gamma_m, M^\gamma(\Gamma_k, dz))|\widehat \Gamma_n \right] = \int_D f(z)M^a(\tilde \Gamma_n, M^\gamma(\Gamma_n,dz)) =: \widetilde M_{n,n}(f).$$
And in particular then $\E\left[\widetilde M(f) | \widehat \Gamma_n\right] = \widetilde M_{n,n}(f)$. But $\widetilde M_{n,n}(f)$ is then a regular martingale w.r.t. $\widehat \Gamma_n$ and thus converges almost surely and in $L^1(d\widehat \Gamma)$ to $\widetilde M$. Hence the measurability follows. As the expectation of $\widetilde M(dz)$ w.r.t the pair $(\Gamma, \tilde \Gamma)$ is equal to Lebesgue measure, it now also follow that the expectation of $\widetilde M(dz)$ w.r.t $\widehat \Gamma$ is equal to Lebesgue measure.

Thus to prove the proposition, it suffices to further check the condition \eqref{eq:char2} with $\widehat \Gamma$ in the role of $\Gamma$ and $\sqrt{\gamma^2 + a^2}$ in the role of $\gamma$. But \eqref{eq:char2} holds for $\widetilde M_{n,n}(dz)$ as it has the law of $M_n^{\sqrt{\gamma^2 + a^2}}(\widehat \Gamma,dz)$. Thus as $\widetilde M_{n,n}(dz) \to \widetilde M(dz)$ almost surely, we obtain \eqref{eq:char2} also for $\widetilde M(dz)$ and hence conclude the lemma.

\end{proof}

Thus given a GMC measure of parameter $\gamma_0$, we can construct one with parameter $\gamma_1$ as close as we wish to $\sqrt{\gamma_0^2 + 2-\frac{\gamma_0^2}{2}} = \sqrt{2+\gamma_0^2/2}$. Iterating this we can construct the GMC measure with parameter $\gamma_n$ as close as we wish to $$\sqrt{2+1+...+2^{2-n} +\gamma_0^2\frac{1}{2^n}}.$$ 
Taking $n$ sufficiently large we can obtain any $\gamma < 2$ using such an iteration. Thus this gives us another construction of the GMC measure in the whole $\gamma \in (-2, 2)$ regime. Notice that again there is nothing specific about the GFF in this section, and all the arguments hold for general nice log-correlated fields in 2D and even in higher dimensions.

\begin{rem}
Here we used Corollary \ref{cor:dimsharp} and thus the calculations of Lemma \ref{lem:2ndme} to justify the dimension of the GMC measure, and thus as presented we did not really give an alternative construction in $L^1$ regime. However, the dimension of $M^\gamma(dz)$ in the $L^2$ regime could also be obtained by different means than by removing thick points, for example by using rooted measures. Thereafter one would bootstrap the dimension together with enlarging the interval for $\gamma$, providing a stand-alone method to treat the $L^1$ regime. And after all, the argument also seems fun enough to present it for its own sake.
\end{rem}

\subsection*{References to the literature}

The fact that the GMC measure can be seen as a Radon-Nikodym derivative seems to stem from \cite{Sato}, and is developed in larger generality in \cite{Sha}. The idea of characterizing GMC as the exponential function (Proposition \ref{prop:chr2}) seems to first appear in \cite{Sha}, from where we also took the proof of the second characterization. 

The observation that by defining GMC on top of GMC one can bypass the $L^2$-barrier can be traced back to \cite{Kah}. In fact, he uses this idea rather to obtain the exact lower bound on the dimension of the measure, as his method first only applies in the $L^2$ regime. 

\section{Behaviour of the measure on different scales: scaling relations and multifractality}\label{sec:scaling}

One of the interesting aspects of the GMC measures is its non-trivial scaling. This scaling is for example expressed in the fact that the moments of balls of radius $r$, i.e. $\E M(\Gamma,B_r)^q$, that for Lebesgue measure scale as $r^{2q}$, here scale as $r^{\zeta(q)}$ with $\zeta(q)$ quadratic in $q$. More precisely we have the following scaling relation:

\begin{prop}[Scaling relation for GMC]\label{prop:SR}
Let $R \in (0,1)$ be fixed. and let $\Gamma_\eps$ be the circle average approximation of the GFF $\Gamma$. Then for any $z \in R\D$  and $r, \eps > 0$ with $r \geq 2\eps$ are sufficiently small, we have 
$$\E \left[M(\Gamma_\eps, B_r(z))^q\right] \asymp r^{\zeta(q)}\E \left[M(\Gamma_{\frac{\eps}{2r}}, \frac{1}{2}\D)^q\right],$$
where $\zeta(q) = (2+\frac{\gamma^2}{2})q - \frac{\gamma^2 q^2}{2}$. In particular, as soon as we know that the $q$-th moment of $M(\Gamma_\eps,\D)$ is uniformly bounded for all $\eps \geq 0$, we have for all $\eps \geq 0$ and, in particular, for the limiting measure 
$$\E \left[M(\Gamma_\eps, B_r(z))^q\right] \asymp_q r^{\zeta(q)}.$$
\end{prop}

This should be compared with exponential moments of a Gaussian: let $X_r$ be a zero mean Gaussian of variance $-\log r$. Then
$$\E e^{q(\gamma X_r - \frac{\gamma^2}{2}\E X_r^2)} = r^{\frac{\gamma^2}{2}q - \frac{\gamma^2 q^2}{2}}.$$
And indeed, as we will see, one way to deduce these scaling relations is to use a relation between the mass of $r$- balls and the exponential of a Gaussian of variance $-\log r$. The remaining $2q$ term in $\zeta(q)$ then just comes from the fact that Lebesgue measure of an $r$-ball is $r^2$. Indeed,
\begin{itemize}
\item at the scale $r$, we see a scaled copy of the measure times $\exp(\gamma X_r- \frac{\gamma^2}{2}\E [X_r^2])$ where $X_r$ is an independent Gaussian shift of variance $-\log r$, stemming from the $r$-circle average process of the underlying field.
\end{itemize}

To make this precise we have the following comparison Lemma. Here, as in Section \ref{sec:prelim}, for some fixed ball $B$, we decompose the GFF as a sum of a zero boundary GFF in the ball $B$, denoted $\Gamma^B$ and a random harmonic function in this disk, denoted $h_B$. 

\begin{lemma}[GMC mass vs GFF circle average]\label{lem:cmp}
Let $R \in (0,1)$ be fixed. Then for any $\eps, r, \tilde r \in (0,1)$ with $0.5\tilde r \geq r \geq \eps$ small enough, and for all $z \in R\D$, we have that
$$\tilde r^{\frac{\gamma^2}{2}}e^{\gamma I_{\tilde r, r,z} }M(\Gamma_\eps^{B_{\tilde r}(z)}, B_r(z)) \lesssim M(\Gamma_\eps, B_{r}(z)) \lesssim \tilde r^{\frac{\gamma^2}{2}}e^{\gamma S_{\tilde r,r,z}} M(\Gamma_\eps^{B_{\tilde r}(z)}, B_r(z)),$$
where $I_{\tilde r, r,z}$ and $S_{\tilde r, r,z}$ are respectively the infimum and supremum over $w \in B_{r}$ of the process $h_{B_{2\tilde r}}(w)$. In particular, they are independent of $\Gamma^{B_{\tilde r}}$ and are $-\log \tilde r$ sub-Gaussians with mean of order $O(\sqrt{r/\tilde r})$.

Moreover, if the limit of $M(\Gamma_\eps, dz)$ exists as $\eps \downarrow 0$, we have the inequality without $\eps$.
\end{lemma}

Before proving this lemma, let us see how it implies Proposition \ref{prop:SR}:

\begin{proof}[Proof of Proposition \ref{prop:SR}]
Let us just prove the upper bound, the lower bound following similarly.
From Lemma \ref{lem:cmp} with $\tilde r = 2r$, we deduce that
$$\E \left[M(\Gamma_\eps, B_r(z))^q\right] \leq C r^{\frac{q\gamma^2}{2}}\E\left[e^{q\gamma S_{2r,r,z}} M(\Gamma_\eps^{B_{2r}(z)}, B_r(z)))^q\right].$$
Moreover from the same Lemma, $S_{2r, r, z}$ is $-\log r$ sub-Gaussian of finite mean, and independent of $\Gamma^{B_{2r}}$. Thus
$$\E \left[M(\Gamma_\eps, B_r(z))^q\right] \lesssim r^{\frac{\gamma^2(q-q^2)}{2}} \E \left[M(\Gamma_\eps^{B_{2r}(z)}, B_r(z)))^q\right].$$
But by the scale-invariance of the Gaussian free field $\Gamma_\eps^{B_{2r}(z)}$ has the same law as as $\Gamma_{\frac{\eps}{2r}}^{\D}$ and thus by change-of-coordinates Lemma \ref{lem:coc} the claim follows.
\end{proof}

The lemma is also quite direct:

\begin{proof}[Proof of Lemma \ref{lem:cmp}]
Let us just prove the upper bound, the lower bound follows again similarly.
From the Markov decomposition of the GFF for the ball $B_{\tilde r}$, we can write the GFF $\Gamma$ restricted to $B_r = B_r(z)$ as the independent sum of $\Gamma^{B_{\tilde r}}$ and $h_{B_{\tilde r}}$. Now for $w \in B_r$ and $\eps \leq r$, as $h_{B_{\tilde r}}$ is harmonic, we have that its $\eps$ average around $w$ is just equal to $h_{B_{ \tilde r}}(w)$. Thus for $w \in B_r(z)$ we have that $\Gamma_\eps(w) \leq \Gamma_\eps^{B_{\tilde r}}(w) + \sup_{\tilde w \in B_r} h_{B_{\tilde r}}(\tilde w)$. 

Now notice that also $\G_{\D}(w_1,w_2) \leq \G_{B_{2\tilde r}}(w_1,w_2) - \log \tilde r - c_z$ for $w_1,w_2 \in B_{2\tilde r}(z)$ and some constant $c_z$ depending on $z$ and uniform over $z \in R\D$. Plugging this into the definition of $M(\Gamma_\eps^{B_{2\tilde r}(z)}, B_r(z))$ we obtain the inequality.

It remains to argue that $S_{\tilde r, r,z}$ is $-\log \tilde r$ sub-Gaussian with mean order $O(\sqrt{r/\tilde r})$. Sub-Gaussianity follows from the Borell-TIS inequality, stated in Theorem \ref{thm:TIS}, as $h_{B_{\tilde r}}(w)$ inside $B_r$ is a almost surely bounded Gaussian process for which the variance of $h_{B_{\tilde r}}(w)$ is bounded by $-\log \tilde r$ for each $w \in B_r$.

The finite mean part follows from the Dudley entropy bound and the continuity estimate on the circle-average process. Indeed, we have that
$$\E\left[(h_{B_{\tilde r}}(w_1) - h_{B_{\tilde r}}(w_2))^2\right] \leq \E\left[(\Gamma_{\tilde r}(w_1) - \Gamma_{\tilde r}(w_2))^2\right].$$
By Lemma \ref{lem:ca} this is bounded by $\frac{|z-w|}{\tilde r}$. But then a direct calculation shows that covering number for the Gaussian pseudo-metric $d(w_1, w_2) = \sqrt{\E\left[(h_{B_{\tilde r}}(w_1) - h_{B_{\tilde r}}(w_2))^2\right]}$ for the process $h_{B_{2 \tilde r}}$ inside $B_r$ is given by 
$$N(\delta,B_r,d) \asymp \frac{r^2}{\tilde r^2\delta^4},$$
where the implied constants are universal. Evaluating the entropy integral and using Theorem \ref{thm:Dudley} gives that $\E S_{\tilde r, r, z} \leq c\sqrt{r/\tilde r}$.

\end{proof}

We will see that there are several other nice consequences to this lemma.

\subsection{Negative moments}

Firstly, Lemma \ref{lem:cmp} allows to control the lower tails of $M(\Gamma_\eps, \D)$.

\begin{lemma}\label{lem:tail}
Let $\Gamma_\eps$ be as in Theorem \ref{thm:basic} and $\gamma < 2$. Then there are $c, C > 0$, so that for all $\eps >0$ sufficiently small and all $\nu > 0$ $$\P\left(\log M(\Gamma_\eps, \D) \leq - \nu\right) \leq ce^{-C\nu^2}.$$
\end{lemma}

\begin{proof}
It is slightly more convenient to work with $p(\nu) := \P\left(\log M(\Gamma, \frac{1}{2}\D) \leq - \nu\right)$. Suppose that $\nu$ is large enough, say $\nu \geq 1000$. The idea is that for this mass to be small, the mass in both disks of radius $1/8$ around $\pm 1/4$ has to be at least as small. Lemma \ref{lem:cmp} and the Markov property help us to turn this into a useful recursive relation. Indeed, we can write 
$$p(\nu) \leq \P\left[\left(\log M(\Gamma, B_{1/8}(1/4)) \leq - \nu\right)\cap\left(\log M(\Gamma, B_{1/8}(-1/4)) \leq - \nu\right) \right].$$
Now Lemma \ref{lem:cmp} with $\tilde r = 1/4$, $r = 1/8$ and $z = 1/4$ in the limiting case $\eps = 0$, gives that for some absolute constant $C_0$
$$C_0 + \gamma I_{1/8,1/8,1/4} + \log M(\Gamma_\eps^{B_{1/4}(1/4)}, B_{1/4}(1/4)) \leq \log M(\Gamma, B_{1/8}(1/4)).$$
Thus if 
$$\{\log M(\Gamma, B_{1/8}(1/4)) \leq - \nu\}\cap\{\log M(\Gamma, B_{1/8}(-1/4)) \leq - \nu\}$$
happens then either
\begin{enumerate}
\item $I \leq -0.1\nu$, where $I =  \gamma (I_{1/4,1/8,1/4} \wedge I_{1/4,1/8, -1/4})$ is a sub-Gaussian; or
\item $I \geq -0.1\nu$, but we have
$$\{\log M(\Gamma^{B_{1/4}(1/4)}, B_{1/8}(1/4)) \leq -0.85\nu\}\cap\{\log M(\Gamma^{B_{1/4}(-1/4)}, B_{1/8}(-1/4)) \leq -0.85\nu\}$$
where we have absorbed the constant $C_0$ in the term $-0.85\nu$.
\end{enumerate}

The probability of the first term is bounded by $e^{-C_1\nu^2}$ for some $C_1 > 0$. By the change-of-coordinates formula $M(\Gamma^{B_{1/4}(1/4)},B_{1/8}(1/4))$ has the same law as as $M(\Gamma,\frac{1}{2}\D)$ up to some fixed constant. Thus, using further the Markov property, we can bound the probability of the event in (2) by $p(0.8\nu)^2$. Hence for all $\nu \geq 1000$,
$$p(\nu) \leq e^{-C_1\nu^2} + p(0.8\nu)^2.$$
But now consider $\nu_m = 1000*(1.25)^m$. Then suppose for contradiction that $p(\nu_m) \geq e^{-a\nu_m^2}$ with $a < C_1$ for infinitely many $m$. Then from the relation above it follows that $p(\nu_{m-1}) \geq e^{-0.9a\nu_{m-1}^2}$. Iterating, $p(1000) \geq e^{-(0.9)^{m+1}a1000^2}$. But this converges to $1$ as $m \to \infty$, giving a contradiction with $p(1000) < 1$. 
\end{proof}

This immediately implies the existence of all negative moments:

\begin{prop}\label{prop:nm}
Suppose $M(\Gamma_\eps, dz)$ converges to some $M(\Gamma, dz)$ weakly in law. Then for all $\eps >0$ sufficiently small and all $m \geq 0$, we have that $\E \left[M(\Gamma_\eps, \D)^{-m}\right]$ is bounded uniformly in $\eps$.
\end{prop}

\subsection*{References to the literature}

The scaling relation, Proposition \ref{prop:SR} for the GMC measure is thoroughly discussed in the review \cite{RVrev}, is already derived in \cite{RV}, but goes back much earlier in the special case of multiplicative cascades. As here, the key is to derive the moment estimates for one convenient field and then use Kahane's inequalities to deduce them for all sufficiently nice log-correlated fields. 

The local relation between the exponential of the circle average of the GFF and the GMC mass plays a key role in \cite{DS}, although it is not stated as explicitly as Lemma \ref{lem:cmp}. The argument for proving negative moments presented here comes from \cite{DS}, but there are other possible options: e.g. there is a cute Bootstrap argument originally used on cascades in \cite{Molchan}, then applied in \cite{RV} in the case of GMC measures and explained in detail in an appendix of \cite{RCW}.

\section{Measurability of the GFF w.r.t. the GMC measure at fixed $\gamma$}\label{sec:meas}

Another nice consequence of the philosophy of the last section, i.e. that at the scale $r$, we see a scaled copy of the measure + an independent Gaussian shift of variance $-\log r$, is the measurability of the GFF w.r.t. the GMC measure at some fixed $\gamma$. The first lemma is a one-point estimate and provides another mathematical statement for the heuristic that $e^{\gamma \Gamma_\eps(z)} \eps^{2+\frac{\gamma^2}{2}} \sim M(\Gamma, B_\eps(z))$:

\begin{lemma}[Local comparison]\label{lem:lcl}
Let $R \in (0,1)$ be fixed. Also let $\Gamma_\eps$ be the circle average approximation of the GFF $\Gamma$ and $\gamma < 2$. Then there is some $C > 0$ such that for any $z \in R\D$  and $\eps > 0$ with sufficiently small, we have 
$$\E\left[\left(\Gamma_\eps(z) + (\gamma/2 + 2/\gamma)\log \eps - \frac{1}{\gamma} \log M(\Gamma, B_{\eps}(z)) \right)^2\right] < C.$$
\end{lemma}

Observe that up to an additive constant $(\gamma/2 + 2/\gamma)\log \eps$ is just the mean of $\log M(\Gamma, B_{\eps}(z))$. 

\begin{proof}
From Lemma \ref{lem:cmp} with $0.5\tilde r = r = \eps$ we see that it suffices to prove that 
$$\E\left[\left(\Gamma_\eps(z) + 2\log \eps - I_{2\eps,\eps,z} - \frac{1}{\gamma} \log M(\Gamma^{B_{2\eps}}, B_{\eps}(z)) \right)^2\right] < C,$$
and the same inequality with $S_{2\eps,\eps,z}$ in place of $I_{2\eps,\eps,z}$. We will only prove the former, as the proof is exactly the same.

From the same lemma, we know that $E_1 := I_{\eps,\eps,z} - h_{B_{2\eps}}(z) = \inf_{w \in B_\eps} h_{B_{2\eps}}(w) - h_{B_{2\eps}}(z)$ is a $100-$sub-Gaussian of finite mean independent of $\Gamma^{B_{2\eps}}$. Also, $E_2 := \Gamma_\eps(z) - h_{B_{2\eps}}(z)$ is a Gaussian of bounded variance of mean zero, say $E_2$. Thus
$ \E\left[(E_1 + E_2)^2\right] < C'.$
Hence bounding 
$$\E\left[\left(E_1 + E_2 + 2\log \eps - \log  M(\Gamma^{B_{2\eps}}, B_{\eps}(z)) \right)^2\right]$$
by
$$2\E\left[(E_1 + E_2)^2\right] + 2\E\left[\left(2\log \eps - \log  M(\Gamma^{B_{2\eps}}, B_{\eps}(z)) \right)^2\right],$$ we are left to show that the latter is bounded by some constant.

But we know that $M(\Gamma^{B_{2\eps}(z)}, B_{\eps}(z))$ has the law of $\eps^2 M(\Gamma_{1/2},\frac{1}{2}\D)$. Moreover one can directly verify that $\E\left[\left(\log M(\Gamma^{1/2},\frac{1}{2}\D)\right)^2\right]$ is finite.
\end{proof}

Thus around a point $z$ the $\eps$-circle-average of the GFF and the properly rescaled and centred version of $\log M(\Gamma, B_{\eps}(z)$ are equal up to a random variable of finite variance. In fact, it comes out that the covariance of these differences decay very rapidly, giving effectively a field that has at any point bounded variance and is independent at distance $\geq \eps$. Such a field converges to $0$ in the space of distributions, showing that the properly rescaled and centred version of $\log M(\Gamma, B_{\eps}(z))$ converges to the GFF itself.
\begin{prop}\label{prop:mes}
Let $\Gamma_\eps$ be the circle average approximation of the GFF $\Gamma$ and $\gamma < 2$ and let $m_\eps(z) = \frac{1}{\gamma}\log M(\Gamma, B_{\eps}(z)$. Then $\Gamma_\eps(z) - m_\eps(z) + \E m_\eps(z)$ converges to $0$ in the space of distributions in probability. In particular $\Gamma$ is measurable w.r.t. $M^\gamma(\Gamma,dz)$ for any fixed $\gamma$ for which $M^\gamma(\Gamma_\eps,dz)$ is non-trivial.
\end{prop}

The correlation estimate is as follows:

\begin{lemma}\label{lem:cor}
Let $R \in (0,1)$ be fixed, $\Gamma_\eps$ be the circle average approximation of the GFF $\Gamma$ and $\gamma < 2$. Denote $m_\eps(z) = \frac{1}{\gamma} \log M(\Gamma, B_{\eps}(z))$. Then for any $z,w \in R\D$, of distance larger than $2\sqrt{\eps}$, we have that
$$\E\left[\left(\Gamma_\eps(z) - m_\eps(z) + \E m_\eps(z)\right)\left(\Gamma_\eps(w) - m_\eps(w) + \E m_\eps(w)\right)\right] < C\eps^{1/5}.$$
\end{lemma}

Let us show how this implies the proposition. 

\begin{proof}[Proof of Proposition \ref{prop:mes}]
Denote $R_\eps(z) := \Gamma_\eps(z) - m_\eps(z) + \E_m\eps(z)$. Notice that by taking $R \to 1$, it suffices to show $R_\eps(z)$ converges to zero as a distribution on $R\D$ for any fixed $R < 1$.
 
To do this we just bound $\E (R_\eps(z), f)^2$ by $$\int_{|z-w| \geq 2\sqrt{\eps}} \E\left[R_\eps(z) R_\eps(w)\right] f(z)f(w) dzdw + \int_{|z-w| \leq 2\sqrt{\eps}} \E\left[R_\eps(z)R_\eps(w)\right] f(z)f(w) dzdw.$$
The first term is bounded by a constant times $\eps^{1/5}$ by Lemma \ref{lem:cor}. In the second term, we use Cauchy-Schwartz to bound 
$$\E\left[R_\eps(z)R_\eps(w)\right]^2  \leq \E\left[R_\eps(z)^2\right]\E\left[R_\eps(w)^2\right],$$
and thus applying Lemma \ref{lem:lcl} we see that the integrand is bounded by some absolute constant. Thus the second term is of order $\eps$ and the proposition follows.
\end{proof}

The proof of the correlation estimate also follows using a version of Lemma \ref{lem:cmp}:

\begin{proof}[Proof of Lemma \ref{lem:cor}]
Recall that $\E \Gamma_\eps(z) = 0$. Hence it suffices to prove that
$$\E\left[\left(m_\eps(z)-\Gamma_\eps(z)\right)\left(m_\eps(w)-\Gamma_\eps(w)\right)\right] < \E m_\eps(z)\E m_\eps(w) +C\eps^{1/5}.$$
Roughly speaking, this says that all the correlation between $m_\eps(z)$ and $m_\eps(w)$ is contained in the correlation between $\Gamma_\eps(w)$ and $\Gamma_\eps(z)$. To show this, we will decompose the terms into correlated and decorrelated parts using a slightly more precise version of the comparison in Lemma \ref{lem:cmp}. Indeed, from the Markov decomposition of the GFF for the ball $B_{\sqrt{\eps}}(z)$, we have:
\begin{equation}\label{eq:correst}
I_{\sqrt{\eps},\eps,z} + \frac{1}{\gamma}\log \tilde M(\Gamma^{B_{\sqrt{\eps}}(z)},B_\eps(z)) \leq m_\eps(z) \leq S_{\sqrt{\eps},\eps,z} + \frac{1}{\gamma}\log \tilde M(\Gamma^{B_{\sqrt{\eps}}(z)},B_\eps(z)),
\end{equation}
where $I_{\sqrt{\eps},\eps,z}$ and $S_{\sqrt{\eps}, \eps, z}$ are as in Lemma \ref{lem:cmp} and $\tilde M(\Gamma, dw) = M(\Gamma, (\frac{\CR(w,D)}{\CR(w,B_{\sqrt{\eps}}(z)})^{\gamma^2/2}dz)$ is a GMC measure w.r.t. a finite perturbation of Lebesgue measure.
Thus we can write
$$m_\eps(z) -\Gamma_\eps(z) \leq \tilde S_{\sqrt{\eps},\eps,z} + \frac{1}{\gamma}\log \tilde M(\Gamma^{B_{\sqrt{\eps}}(z)},B_\eps(z)) - \Gamma_\eps^{B_{\sqrt{\eps}}(z)}(z),$$
where $\tilde S_{\sqrt{\eps},\eps, z} = S_{\sqrt{\eps},\eps,z} - h_{B_{\sqrt{\eps}}(z)}(z)$.
Now bound $\E\left[(m_\eps(z)-\Gamma_\eps(z))\left((m_\eps(w)-\Gamma_\eps(w))\vee 0\right)\right]$ by $$\E\left[(\tilde S_{\sqrt{\eps},\eps,z} +\frac{1}{\gamma}\log \tilde M(\Gamma^{B_{\sqrt{\eps}}(z)},B_\eps(z)) - \Gamma_\eps^{B_{\sqrt{\eps}}(z)}(z))\left((m_\eps(w)-\Gamma_\eps(w))\vee 0\right)\right].$$
Now by Cauchy-Schwarz
$$\E\left[\tilde S_{\sqrt{\eps},\eps,z}\left((m_\eps(w)-\Gamma_\eps(w))\vee 0\right)\right] \leq \sqrt{\E \left[\tilde S_{\sqrt{\eps},\eps,z}^2\right]\E \left[\left((m_\eps(w)-\Gamma_\eps(w))\vee 0\right)^2\right]}.$$
From Lemma \ref{lem:cmp} we know the mean of $\tilde S_{\sqrt{\eps},\eps,z}$ is bounded by $c\eps^{1/4}$. By Lemma \ref{lem:ca}, 

and by Borell-TIS inequality, its variance is also bounded by $\sqrt{\eps}$.
Thus the first term is bounded by $\eps^{1/4}$ . The second term is bounded by a constant by Lemma \ref{lem:lcl}.

On the other hand, $\frac{1}{\gamma}\log \tilde M(\Gamma^{B_{\sqrt{\eps}}(z)},B_\eps(z)) - \Gamma_\eps^{B_{\sqrt{\eps}}(z)}(z)$ is independent of $(m_\eps(w)-\Gamma_\eps(w))$. Moreover, from the lower bound in Equation \eqref{eq:correst} and the estimate on the mean of $I_{\sqrt{\eps},\eps,z}$ we have
$$\E\left[(\frac{1}{\gamma}\log \tilde M(\Gamma^{B_{\sqrt{\eps}}(z)},B_\eps(z)) - \Gamma_\eps^{B_{\sqrt{\eps}}(z)}(z))\right] \leq \E m_\eps + c\eps^{1/4}.$$
Hence we can bound
$$\E\left[(\frac{1}{\gamma}\log \tilde M(\Gamma^{B_{\sqrt{\eps}}(z)},B_\eps(z)) - \Gamma_\eps^{B_{\sqrt{\eps}}(z)}(z))\left((m_\eps(w)-\Gamma_\eps(w))\vee 0\right)\right]$$
by
$$\E m_\eps(z)\E\left[\left((m_\eps(w)-\Gamma_\eps(w))\vee 0\right)\right]+c'\eps^{1/5},$$
where $c'$ is some finite constant. Thus (redefining the constant $c$)
$$\E\left[(m_\eps(z)-\Gamma_\eps(z))\left((m_\eps(w)-\Gamma_\eps(w))\vee 0\right)\right] \leq c\eps^{1/5} + \E m_\eps(z)\E\left[\left((m_\eps(w)-\Gamma_\eps(w))\vee 0\right)\right].$$
We get similarly, by considering the second inequality in Lemma \ref{lem:ca}, that
$$\E\left[(m_\eps(z)-\Gamma_\eps(z))\left((m_\eps(w)-\Gamma_\eps(w))\wedge 0\right)\right] \leq c\eps^{1/5} + \E m_\eps(z)\E\left[\left((m_\eps(w)-\Gamma_\eps(w))\wedge 0\right)\right].$$
And summing together we obtain
$$\E\left[(m_\eps(z)-\Gamma_\eps(z))(m_\eps(w)-\Gamma_\eps(w))\right] \leq c\eps^{1/5} + \E m_\eps(z)\E\left[(m_\eps(w)-\Gamma_\eps(w))\right].$$
Using again that $\E\Gamma_\eps(w) = 0$, the lemma follows. 
\end{proof}

\subsection*{References to the literature}

Measurability of the GFF w.r.t. the GMC measure at a fixed $\gamma$ comes from \cite{BerSheSun}. As mentioned above in Section \ref{sec:basic}, the measurability of the GFF given $M^\gamma(\Gamma,dz)$ over a whole interval of $\gamma$ around $0$ can be considered folklore. 

\section{Kahane convexity inequality}\label{sec:Kah}

One of the key tools in GMC theory is a comparison inequality, reminding vaguely Slepian type of inequalities for Gaussian processes both in its statement and proof. It is mostly used to transform results known for a GMC measure w.r.t. some Gaussian fields, to other fields with a slightly different, but comparable covariance kernel. For example, using these inequalities one could transform many of the results of this paper, to other $\log$-correlated fields in two dimensions. 

\begin{prop}[Kahane convexity inequality]\label{prop:KCI}
Let $\Gamma_1$, $\Gamma_2$ be two point-wise defined Gaussian fields on $D$ such that their covariance kernels satisfy $\C_1(z,w) \leq \C_2(z,w)$ for any $z,w \in D$. Then for any convex function $f$ defined on $\R^+$, we have that 
$$\E \left[ f(M(\Gamma_1, D))\right] \leq \E \left[ f(M(\Gamma_2, D))\right].$$
Moreover, the opposite inequality holds for concave functions.
\end{prop}

The proof is quite slick, but the basic strategy is simple: just to verify that the corresponding expectation increases, if we increase the covariance kernel infinitesimally. 

\subsection*{Instructive case} Let us first treat a very simple case, which however provides an outline for the proof:
\begin{itemize}
\item Assume that $\Gamma_2(z) = \Gamma_1(z) + X$ for some independent Gaussian $X$ with variance $\C_X$. 
\end{itemize}
In this case we can interpolate between $\Gamma_1$ and $\Gamma_2$ using $Y_\lambda = \Gamma_1 + \lambda X$ with $\lambda \in [0,1]$. Set $F(\lambda) = \E \left[ f(M(Y_\lambda, D)\right]$. It suffices to prove that $F'(\lambda) \geq 0$. Writing out we have
$$F'(\lambda) = \E \left[\int_D dz (X - \lambda \C_X e^{\lambda X - \frac{\lambda^2}{2}\C_X}e^{\Gamma_1(z) - \frac{\lambda^2}{2}\C_1(z,z)}f'(M(Y_\lambda, D))\right].$$
We can take the expectation inside the integral. Using the independence of $X$ and $\Gamma_1$ we can write the integrand for fixed $z$:
$$\E_{\Gamma_1}\left[e^{\Gamma_1(z) - \frac{\lambda^2}{2}\C_1(z,z)}\E_X\left((X - \lambda \C_X)e^{\lambda X - \frac{\lambda^2}{2}\C_X}f'(M(Y_\lambda, D))\right)\right],$$
where by $\E_X, \E_{\Gamma_1}$ we mean taking expectation w.r.t. only $X$ or $\Gamma_1$ respectively.

It thus suffices to show that $\E_X\left((X - \lambda \C_X)e^{\lambda X - \frac{\lambda^2}{2}\C_X}f'(M(Y_\lambda, D))\right) \geq 0$. For fixed $\Gamma_1$, by convexity of $f$ the function $f'(M(Y_\lambda, D))$ is increasing with the value of $X$. On the other hand $\E_X\left[(X - \lambda \C_X)e^{\lambda X - \frac{\lambda^2}{2}\C_X}\right] = 0$ as it is equal to the derivative in $\lambda$ of $\E_X e^{\lambda X - \frac{\lambda^2}{2}\C_X}$. Thus we conclude with the following elementary claim:

\begin{claim}\label{clm:elem}
Let $X$ be a zero mean random variable and $f:\R \to \R$ be increasing. Then $\E Xf(X) \geq 0$.
\end{claim}

\subsection*{General case} In the general case one cannot expect to have such a simple interpolation. However, one can always interpolate from $\Gamma_1$ to $\Gamma_2$ via, say, $$Y_\lambda(z) = \sqrt{1-\lambda} \Gamma_1(z) + \sqrt{\lambda} \Gamma_2(z).$$
Let $\C_\lambda(z,w)$ be the covariance of $Y_\lambda(z)$. Write also $Z_\lambda(z) = \frac{d}{d\lambda}Y_\lambda(z)$. We then have:

$$F'(\lambda) = \E \left[\int_D dz (Z_\lambda(z) - \E Y_\lambda(z) Z_\lambda(z)) e^{Y_\lambda(z) - \frac{1}{2}\C_\lambda(z)}f'(M(Y_\lambda, D))\right],$$
and we want to show that $F'(\lambda) \geq 0$.
To do this, we want to again bring the expectation inside the integral and integrate over $Z_\lambda(z)$ for some fixed $z$. In this respect, write $$Y_\lambda(w) = c_zZ_\lambda(z) + W_\lambda(w),$$ where $W$ is an Gaussian field on $D$ with covariance $\C_W$ and independent of $Z_\lambda(z)$, and $c_z \E Z_\lambda(z)^2 = \E Z_\lambda(z) Y_\lambda(w)$. 

Denoting $Z_\lambda(z) = X$ we can write the integrand for fixed $z$ as
$$\E_{W_\lambda}\left[e^{W_\lambda(z) - \frac{1}{2}\C_W(z,z)}\E_X\left((X - c_z\C_X)e^{\lambda X - \frac{c_z^2}{2}\C_X}f'(M(Y_\lambda, D))\right)\right].$$ 
Again it suffices to show that $\E_X\left((X - c_z\C_X)e^{\lambda X - \frac{c_z^2}{2}\C_X}f'(M(Y_\lambda, D))\right) \geq 0$.
As $\E Z_\lambda(z) Y_\lambda(w) = \frac{1}{2}(\C_2(z,w) - \C_1(z,w)) \geq 0$, we deduce that $f'(M(Y_\lambda,D))$ is increasing with $X$. Thus as also $\E_X\left[(X - c_z\C_X)e^{\lambda X - \frac{c_z^2}{2}\C_X}\right] = 0$, we can again apply Claim \ref{clm:elem} to conclude.
\subsection{Positive moments and the multiscale approach to uniform integrability}\label{sec:UI1}

We now show how to obtain uniform integrability in the region $\gamma \in [\sqrt{2}, 2)$. In fact we will see that $\E M(\Gamma,f)^q < \infty$ if $q < 4/\gamma^2$. The idea is to use local scaling relations and near independence of far-away regions. 

\begin{prop}\label{prop:posm1}[Existence of positive moments]
Let $D$ be a finite domain, fix $\gamma < 2$ and let $\Gamma_\eps$ be the circle-average approximation. Then for any $0 < q < 4/\gamma^2$, there is some $C_q$ such that for all $\eps \geq 0$ we have $\E \left[M(\Gamma_\eps,D)^q\right] < C_q$.
\end{prop}

In fact one can also show that this moment bound is sharp, i.e. that the moments for $q \geq 4/\gamma^2$ are all infinite.

\begin{proof} Fix $1 < q < 4/\gamma^2$. For convenience we will this time work on the unit square $S = [0,1]^2$.  
Fix a parameter $n \in 2\N$ very large, but finite, to be determined later and let $0 < \eps < 0.5n^{-1}$.. Tile the square into $n^2$ little squares $S_i$ of side-length $n^{-1}$ and consider the 4 disjoint
unions of squares that contain every other little square both vertically and horizontally. Call these unions $A_1,
A_2, A_3, A_4$. As before, denote $M_\eps(B) = M(\Gamma_\eps, B)$ for any Borelian $B$. Then by inequality of means and symmetry:
$$ \E \left[M_\eps(S)^q\right] \leq 10\E \left[M_\eps(A_1)^q\right].$$
We will now work with $A_1$ only. We have
\[ \E \left[M_\eps(A_1)^q\right] =\E \left[\left(\sum_{S_i \cap A_1
   \neq \emptyset}M_\eps(S_i)\right) \left( \sum_{S_j \cap A_1 \neq \emptyset} M_\eps
   (S_j) \right)^{q-1}\right] \]
Further for any fixed $i$:
\[ \left( \sum_{S_j \cap A_1 \neq \emptyset} M_\eps (S_j) \right)^{q-1}
   \lesssim M_\eps(S_i)^{q-1} + \left(\sum_{S_j \cap A_1 \neq \emptyset, S_j \neq S_i} M_\eps (S_j)\right)^{q-1}. \]
Thus
\[ \E \left[M_\eps(A_1)^{q}\right] \lesssim \sum_{S_i \cap A_1 \neq
   \emptyset} \E \left[M_\eps (S_i)^{q}\right] + \sum_{S_i \cap A_1 \neq
   \emptyset} \E \left[M_\eps (S_i) \left(\sum_{S_j \cap A_1 \neq \emptyset, S_j \neq S_i} M_\eps (S_j)\right)^{q-1}\right], \]
Now for any $S_i \subset A_1$ let $\partial B_i$ be the boundary of the square of side-length $1.5/n$, that has the same centre $z_i$ as $S_i$. Then $\partial B_i$ separates $S$ into a open square containing $S_i$ and an open set containing all the other squares. Denote by $B_i = S \backslash \partial B_i$. By a slight generalization of Lemma \ref{lem:cmp}, we have
   $$M_\eps(S_i) \leq Cn^{-\frac{\gamma^2}{2}}e^{\gamma S_{1.5/n,1/n,z_i}} M(\Gamma_\eps^{B_i},S_i),$$
   where $S_{1.5/n,1/n,z_i}$ are sub-Gaussians of variance $\log n + c$ and bounded mean.
   But the term $M(\Gamma_\eps^{B_i},S_i)$ is independent of $\left(\sum_{S_j \cap A_1 \neq \emptyset, S_j \neq S_i} M_\eps (S_j)\right)$. Thus
   $$\E \left[M_\eps (S_i) \left(\sum_{S_j \cap A_1 \neq \emptyset, S_j \neq S_i} M_\eps (S_j)\right)^{q-1}\right]\leq Cn^{-\frac{\gamma^2}{2}}\E\left[M(\Gamma_\eps^{B_i},S_i)\right]\E\left[e^{\gamma S_{1.5/n, 1/n, z_i}}M_\eps(A_1)^{q-1}\right].$$
 Using H\"older inequality with $p_1 = 2q-1$ and $p_2 = \frac{2q-1}{2q-2}$ we have
   $$\E\left[e^{\gamma S_{1.5/n, 1/n, z_i}}M_\eps(A_1)^{q-1}\right] \leq \left(\E\left[e^{\gamma(2q-1) S_{1.5/n, 1/n, z_i}}\right]\right)^{1/p_1} \left(\E\left[M_\eps(A_1)^{q-1/2}\right]\right)^{1/p_2}.$$ 
   The first term is just an exponential moment of a sub-Gaussian and thus is finite. The second term is finite as long as $1 < q < 3/2$, as the mean of $M_\eps(A_1)$ is finite. On the other hand $\E\left[M(\Gamma_\eps^{B_i},S_i)\right] = n^{-2}$, and thus for some constant $C_n$ depending on $n$ but not $\eps$ we have:
\[ \E \left[M_\eps(A_1)^{q}\right] \leq \sum_{S_i \cap A_1 \neq
   \emptyset} \E \left[M_\eps (S_i)^{q}\right] + C_n.\]
Moreover by a slight generalization of Proposition \ref{prop:SR} to squares and by noticing that the upper bound also holds for squares near the boundary, we have
\[ \E \left[M(\Gamma_\eps, S_i)^q\right] \lesssim n^{-\zeta(q)} \E \left[M(\Gamma_{n\eps/2},S)^q\right]. \]
As using Kahane's inequality again we have $\E \left[M(\Gamma_{n\eps/2},S)^q\right] \leq \E \left[M(\Gamma_{\eps},S)^q\right]$, we conclude that 
for $q < 4/\gamma^2 \wedge 3/2$, we can bound 
\[ \E \left[M(\Gamma_\eps, S)^q\right] \leq Cn^{2-\zeta(q)} \E \left[M(\Gamma_{\eps/2},S)^q\right] + C_n.\]
But for $\gamma < 2$ and $q < 4/\gamma^2$, we have that $\zeta(q) = (2 + \gamma^2 /
2)q - \gamma^2 q^2 / 2 > 2$ and thus for $n$ large the term $Cn^{2+\gamma^2 (q^2-q) / 2
-2q}$ is smaller than $1$. Hence the claim follows in the case $q < 4/\gamma^2 \wedge 3/2$. 

Finally, notice that we used the condition $q < 3/2$ only when we needed a uniform bound on the $(q-1/2)$-th moment for all $\eps \geq 0$. So after having established the claim for $q =q_0$, we can deduce the existence of moments for all $q < q_0 + 1/2$ and by recursion for all $q <  4/\gamma^2 $.
\end{proof}

In fact, this implies also that the same moments are finite for the orthogonal approximation. Indeed, let $q < 4/\gamma^2$ and recall from Section \ref{sec:basic} that for any continuous bounded $f$ we have $M(\Gamma^n,f) = \lim_{\eps \to 0} \E\left[M(\Gamma_\eps,f)|\Gamma^n\right]$. Thus also $M(\Gamma^n,f)^q = \lim_{\eps \to 0} \E\left[M (\Gamma_\eps,f)|\Gamma^n\right]^q$. Now by Fatou's lemma
$$\E \left[ M(\Gamma^n, dz)^q\right] \leq \liminf_{\eps \to 0} \E \left[ \E\left(M(\Gamma_\eps,f)|\Gamma^n\right)^q\right].$$
From Jensen's inequality $$\E \left[ \E\left(M(\Gamma_\eps,f)|\Gamma^n\right)^q\right] \leq \E \left[ \E\left(M(\Gamma_\eps,f)^q|\Gamma^n\right)\right] = \E\left[M(\Gamma_\eps,f)^q\right].$$
But this is bounded by $C_q$ by the proposition above. Hence in fact Proposition \ref{prop:posm1} also gives 
$$\E\left[ M(\Gamma^n,f)^q\right] \leq C_q$$
and thus convergence of the GMC measure for the orthogonal sequence over the whole sub-critical regime.

\begin{rem}
	Roughly the underlying philosophy of the proof is similar to the treatment of the $L^2$ phase in Section \ref{sec:thickpoints}. There we also ignored correlations at distances larger than $\eps'$, which would compare to $1/n$ in the proof of the proposition above. Moreover, to control the second moment at small distances, we there also effectively use the fact that the field is locally equal to a zero boundary GFF + Gaussian with the variance given by the logarithm of the scale. However, it is not entirely clear how to extract the existence of moments directly from the argument in Section \ref{sec:thickpoints}.
\end{rem}

Using Kahane's convexity inequality, Proposition \ref{prop:nm} and Proposition \ref{prop:SR} we now deduce positive and negative moments for all log-correlated Gaussian fields in 2D.
\begin{cor}[Positive and negative moments for log-correlated fields]
Let $\Phi$ be a Gaussian field on $D$ with covariance kernel $\C_\Phi(z,w)$ of the form $\C_\Phi(z,w) = - \log |z-w| + g(z,w)$, with $g$ continuous and bounded. Then we have that for all $q < 4/\gamma^2$, 
$$\E \left[M(\Phi, B_r(z))^q\right] \asymp_q r^{\zeta(q)}.$$
\end{cor}

\subsection*{References to the literature}

Kahane's convexity inequality and its proof go back to \cite{Kah}. There is also a possible proof going through finite-dimensional Gaussians: in this setting one uses the same interpolation, but concludes using Gaussian integration by parts. Finally one uses continuity to extend the inequality to continuous Gaussian fields. See e.g. appendix of \cite{RV}.

The method of using Kahane's convexity inequality and such a decomposition to prove existence of moments might well be new, at least in such condensed form. The result on moments itself goes back to the seminal article by Kahane \cite{Kah}.

\section{A version of the KPZ relation}\label{sec:KPZ}

In the setting of GMC measures, the KPZ relation can be interpreted as a quadratic equation linking the Euclidean dimension of a fractal set, with its dimension under the GMC measure. There is in fact a plenitude of KPZ relations corresponding to the exact meaning of the dimension both on the Euclidean and the GMC measure side. 

Here we use the so-called scaling dimensions, as it is easy to state, has an intuitive meaning in terms of planar maps and also allows us to use previous techniques.

Let $X$ be some fractal set contained in a compact subset of $\D$.
\begin{itemize}
\item The Euclidean scaling dimension $d_s$ defined (when it exists) through $\Leb(z: B_r(z) \cap X \neq \emptyset)=r^{	d_s + o_r(1)}$.
\item The GMC scaling dimension $q_s$ is defined (when it exists) through $\E M(z: Q_r(z) \cap X \neq \emptyset) = r^{q_s + o_r(1)}$, where $Q_r(z)$ denotes the Euclidean ball around $z$ that satisfies $M(\Gamma,Q_r(z)) = r$. 
\end{itemize}

\begin{prop}[KPZ relation]\label{prop:KPZ}
Let $\gamma < 2$ and suppose that the Euclidean scaling dimension of a compactly supported fractal set $X \subset \D$ exists and is equal to $d_s$. Then the GMC scaling dimension $q$ also exists, and moreover we have the following KPZ relation:
$$d_s = (2-\gamma^2/2)q_s + \gamma^2q_s^2/2.$$
\end{prop}

Notice that if we take $\tilde d = 2 - d_s$ and $\tilde q = 1 - q_s$, then we obtain 
$$\tilde d = \zeta(\tilde q),$$
where $\zeta(\cdot)$ is exactly the same exponent that appeared in Section \ref{sec:scaling}, where we calculated moments of GMC mass of balls of fixed radius:
$$\E \left[M(\Gamma,B_r)^{\tilde q}\right] \asymp r^{\zeta(\tilde q)} = r^{\tilde d}.$$
Moreover, this moment scaling can be turned into a KPZ relation with notions of Minkowski and Hausdorff dimensions on both the Euclidean and the GMC side.

To prove the current KPZ relation at hand, we will do the inverse of the above: we will calculate the moments of the random radius $r$ of the ball $Q_r(z)$, i.e. the radius of the Euclidean ball around $z$ satisfying $M(\Gamma, Q_r(z)) = r$. More precisely, we need to calculate these moments around the typical points of the GMC measure. Indeed, the key lemma is the following:

\begin{lemma}\label{lem:KPZ}
Let $R<1$, $z \in R\D$ and $\tilde Q_r(z) = Q_r(\Gamma + \gamma \G_D(z,\cdot))$. Then
for any positive $q$, $$\E \left[\radS(\tilde Q_r(0))^{\theta(q)}\right] = r^{q+o_r(1)},$$ where $\theta(q) = (2-\gamma^2/2)q + \gamma^2q^2/2$.
\end{lemma}

The proposition follows from this lemma and an auxiliary lemma, saying that this radius doesn't depend too much on the geometric location of the chosen point:
\begin{lemma}\label{lem:KPZaux}
Let $R<1$, $z \in R\D$ and $\tilde Q_r(z) = Q_r(\Gamma + \gamma \G_D(z,\cdot))$. Then for all $s \geq 0$, $$\P\left(\radS(\tilde Q_r(z)) \geq s\right) = \P\left(\radS(\tilde Q_r(0)) \geq s\right)^{1+o_r(1)}.$$
\end{lemma}

Let us first see how this proves Proposition \ref{prop:KPZ} and then prove the lemmas. Notice that the GMC scaling dimensions can be written using the rooted measure from section \ref{sec:RM}. Indeed, by using the two descriptions of the rooted measure, i.e. equation \eqref{eq:char}, we can write
$$\E M(z: Q_r(z) \cap X \neq \emptyset) = \E \int_D \I{Q_r(z) \cap X \neq \emptyset}M(\Gamma,dz) = \E \int_D \I{\tilde Q_r(z) \cap X \neq \emptyset}dz,$$
where as above we denote $\tilde Q_r(z) = Q_r(\Gamma + \gamma \G_D(z,\cdot))$. From Fubini's theorem, we can write this further as $$\int_D \P(\tilde Q_r(z) \cap X \neq \emptyset)dz.$$
By Lemma \ref{lem:KPZaux}, we have
$$\P(\tilde Q_r(z) \cap X \neq \emptyset) = \P\left(\radS(\tilde Q_r(z)) \geq \dist (z,X)\right) = \P\left(\radS(\tilde Q_r(0)) \geq \dist (z,X)\right)^{1+o_r(1)}.$$
Thus by Jensen's inequality and the fact that probabilities are in $[0,1]$, we have
$$\E M(z: Q_r(z) \cap X \neq \emptyset) = \left[\int_D \P\left(\radS(\tilde Q_r(0)) \geq \dist (z,X)\right)dz\right]^{1+o_r(1)}.$$
Now write $d\P_r(dr_q)$ to be the law of $r_q = \radS (\tilde Q_r(0))$. Using Fubini's theorem, we can rewrite 
$$\int_D dz \int_{\R^+} \I{\dist(z,X) \leq r_q)}d\P_r(dr_q) = \int_{\R^+} d\P_r(dr_q) \int_D \I{\dist(z,X) \leq r_q}dz$$
But $$\int_D \I{\dist(z,X) \leq r_q}dz = \Leb(z: B_{r_q}(z) \cap X \neq \emptyset)$$ and by the assumption on the Euclidean dimension
$$\Leb(z: B_{r_q}(z) \cap X \neq \emptyset)) = {r_q}^{d_s+o_1(r_q)}$$
Hence 
$$\E M(z: Q_r(z) \cap X \neq \emptyset) = \left[\E(\radS (\tilde Q_r(0))^{d_s+o_1(r)})\right]^{1+o_r(1)}$$
and the proposition follows from the moment calculation in Lemma \ref{lem:KPZ}.

\begin{proof}[Proof of Lemma \ref{lem:KPZaux}] The basic tool is again a comparison in the style of Lemma \ref{lem:cmp}. To use it, first notice that as $M(\Gamma + \gamma \G_D(z,\cdot), B_r(z))$ is almost surely increasing in $r$,  
$$\P\left(\radS(\tilde Q_r(z)) \geq s\right) = \P\left(M(\Gamma + \gamma \G_D(z,\cdot), B_s(z)) \leq r\right).$$
Now from (a slight generalization of) Lemma \ref{lem:cmp} with $\eps = 0$ and $r = 0.5\tilde r = s$ smaller than $(1-R)/10$, we have (with the notations of that lemma)
$$c s^{\frac{\gamma^2}{2}}e^{\gamma I_{2s,s,z} }M\left(\Gamma^{B_{2s}(z)} + \gamma \G_D(z,\cdot), B_s(z)\right) \leq M\left(\Gamma + \gamma \G_D(z,\cdot), B_s(z)\right)$$
Thus 
$$\P\left(M(\Gamma + \gamma \G_D(z,\cdot), B_s(z)) \leq r\right) \leq \P\left(c s^{\frac{\gamma^2}{2}}e^{\gamma I_{2s,s,z} }M\left(\Gamma^{B_{2s}(z)} + \gamma \G_D(z,\cdot), B_s(z)\right) \leq r\right)$$
Now the law of $\Gamma^{B_{2s}(z)}$ does not depend on $z$, is independent of $I_{2s,s,z}$, and $c_R \G_D(0,\cdot) \leq \G_D(z,\cdot)$ for $z \in R\D$. Thus we further bound by
$$\P\left(c' s^{\frac{\gamma^2}{2}}e^{\gamma I_{2s,s,z} }M\left(\Gamma^{B_{2s}(0)} + \gamma \G_D(0,\cdot), B_s(0)\right) \leq r\right).$$

Finally, write $I_{2s,s,z} = h_{B_{2s}}(z) + \tilde I_{2s,s,z}$ as in the proof of Lemma \ref{lem:cor}. For $z \in R \D$, $h_{B_{2s}}(z)$ are Gaussians of variance $-\log s + v_z$ with $v_z$ bounded over $z \in R \D$, and $\tilde I_{2s,s,z}$ are sub-Gaussians of bounded variance and mean. Thus as $s \to 0$, for any $C$ in some fixed $[c, 1/c]$, we have
$$\P(e^{\gamma I_{2s,s,z} } \leq r_1) \leq \P(e^{\gamma I_{2s,s,0}} \leq Cr_1)^{1-o_s(1)}.$$
The other inequality follows similarly.
\end{proof}

Finally, let us prove the moment estimate:

\begin{proof}[Proof of Lemma \ref{lem:KPZ}]

From the scale-invariance of $\Gamma$ and the fact that $\G_D(0,z) =  - \log |z|$, we have that $M\left(\Gamma^{B_{2s}(0)} + \gamma \G_D(0,\cdot), B_s(0)\right)$ has the law of $(2s)^{2 - \gamma^2} M\left(\Gamma + \gamma \G_D(0,\cdot), B_{1/2}(0z)\right)$. Denoting $M_{1/2} :=M\left(\Gamma + \gamma \G_D(0,\cdot), B_{1/2}(0z)\right)$ we deduce from Lemma \ref{lem:cmp} as above that
\begin{equation}\label{eq:lbKPZ}
M(\Gamma + \gamma \G_D(0,\cdot), B_s(0)) \text{ stochastically dominates } cM_{1/2}\left[(2s)^{2-\gamma^2/2}e^{\gamma h_{B_{2s}}(0)}\right]e^{\tilde I_{2s,s,0}},
\end{equation}
and 
\begin{equation}\label{eq:ubKPZ}
CM_{1/2}\left[(2s)^{2-\gamma^2/2}e^{\gamma h_{B_{2s}}(0)}\right]e^{\tilde S_{2s,s,0}} \text{ stochastically dominates } M(\Gamma + \gamma \G_D(0,\cdot), B_s(0)).
\end{equation}
Here $\tilde I_{2s,s,0}$ is as above, and $\tilde S_{2s,s,0}$ is defined similarly - both are sub-Gaussians of bounded variance of mean that can be bounded independently of $s$. 

Now, in the bounds above only the term in the square brackets depends on $s$ and defines the leading behaviour as $s \to 0$. Following this philosophy, define an auxiliary mass $\widehat M(s) = (2s)^{2-\gamma^2/2}e^{\gamma h_{B_{2s}}(z)}$ and let $\widehat \radS_r := \sup_s \{\widehat M(s) = r\}$. The moment estimate now follows in two steps. We will show that:
\begin{enumerate}
\item $\E \left[\widehat \radS_r^{\theta(q)}\right] \asymp r^q, $ where $\theta(q)$ is as in the statement, and
\item $\E \left[\radS(\tilde Q_r(z))^\lambda\right] =\left(\E \left[\hat \radS_r^\lambda\right]\right)^{1+o_s(1)}$ for all $\lambda \geq 0$ and $z \in R\D$.
\end{enumerate}

The first step just follows from a simple martingale calculation: indeed as $h_{B_{2s}}(0) = \Gamma_{2s}(0)$ and $\Gamma_\eps(0)$ is a Brownian motion when parametrized using $-\log \eps$, then for any $\lambda$, the process $e^{\lambda h_{B_{2s}}(0)}(2s)^{\lambda^2/2}$ is a martingale of mean $1$. When we stop it at $\widehat \radS_r$, on the one hand we have $$(2\widehat\radS_r)^{2-\gamma^2/2}e^{\gamma h_{B_{2\widehat\radS_r}}(0)} = r.$$ But from the optional stopping theorem we also have $\E e^{\lambda h_{B_{2\widehat\radS_r}}(0)}(2\widehat\radS_r)^{\lambda^2/2} = 1.$
Thus $$\E (2\widehat\radS_r)^{\lambda^2/2 + \lambda(\gamma/2-2/\gamma)} = r^{-\lambda/\gamma},$$
and taking $q = -\lambda / \gamma$ gives the claim. 

Using Lemma \ref{lem:KPZaux}, it suffices to show the second step for $z = 0$. 
Let us first argue that $\E \left[\radS(\tilde Q_r(0))^\lambda\right] \geq c \E \left[\widehat \radS_r^\lambda\right]$.
Writing
$\E \left[\radS(\tilde Q_r(0))^\lambda\right] = \E \left[\E \left(\radS(\tilde Q_r(0))^\lambda | \widehat\radS_r\right)\right]$
and using
$$\radS(\tilde Q_r(0))^\lambda \geq \radS(\tilde Q_r(0))^\lambda\I{\widehat \radS_r \leq \radS(\tilde Q_r(0))} \geq \widehat\radS_r^\lambda\I{\widehat \radS_r \leq \radS(\tilde Q_r(0))}$$
we see that 
$$\E \left[\radS(\tilde Q_r(0))^\lambda\right] \geq \E\left[\widehat \radS_r^\lambda \P\left(\radS(\tilde Q_r(0)) \geq \widehat\radS_r | \widehat\radS\right)\right].$$
From $\{\radS(\tilde Q_r(0)) \geq s\} = \{M(\Gamma + \gamma \G_D(0,\cdot), B_s(0)) \leq r\},$
we see that \eqref{eq:ubKPZ} implies
$$\P\left(\radS(\tilde Q_r(0)) \geq s | \widehat\radS_r =s)\right) \geq \P(e^{\tilde S_{2s,s,0}}M_{1/2} \leq 1/C).$$
Now, $M_{1/2}$ is a.s. positive and $\tilde S_{2s,s,0}$ is a sub-Gaussian of bounded variance of mean. Hence this probability is larger than some $c$ uniformly over all $s > 0$ and $\E \left[\radS(\tilde Q_r(0))^\lambda\right] \geq c \E \left[\widehat \radS_r^\lambda\right]$ follows.

It remains to show the other bound. From $\eqref{eq:lbKPZ}$ we deduce that for any sufficiently small $\nu > 0$ it holds that 
$$\P\left(\radS(\tilde Q_r(0)) \geq \widehat \radS_r^{1-\nu}\right) \leq \P\left(e^{\tilde I_{2\widehat \radS_r,\widehat \radS_r,0}}M_{1/2} \leq r^{\nu}/c\right).$$
But from Lemma \ref{lem:tail} and the fact that $M_{1/2} \geq M\left(\Gamma, B_{1/2}(z)\right)$ we see that this probability decays faster than any power. Thus the upper bound also follows. 
\end{proof}

\subsection*{References to the literature}

In the physics literature, the KPZ relation goes back to \cite{KPZ}. The idea to reinterpret this relation as a relation between fractal dimensions goes back to \cite{DS}, giving also the first mathematical derivation of the KPZ relation. See also \cite{Gar} for an introduction to this version and to heuristic in terms of random planar maps. 

Another version of the KPZ relation followed shortly thereafter \cite{RVKPZ}, and it has the advantage of being an almost sure result w.r.t the measure. It uses rather the notion of Hausdorff dimension and is based on the moment calculation of the GMC mass. By now there is a whole multitude of KPZ relations, depending on the notion of dimension used on the Euclidean and the GMC measure side, see e.g. \cite{RVrev} for an overview.

\bibliographystyle{alpha}
\bibliography{bib_GMC}

\end{document}